\documentclass[11pt]{article}

\usepackage{graphicx}
\usepackage{amsmath}
\usepackage{amssymb}
\usepackage[shortlabels]{enumitem}
\usepackage{amsthm}
\usepackage{float}
\usepackage{gensymb}
\usepackage{hyperref}
\usepackage{subcaption}
\usepackage{tikz}
\usetikzlibrary{positioning}
\usetikzlibrary{quotes,angles}
\usepackage[misc,geometry]{ifsym}

\usepackage{fullpage}

\newtheorem{theorem}{Theorem}[section]
\newtheorem{corollary}[theorem]{Corollary}

\newtheorem{proposition}[theorem]{Proposition}

\numberwithin{equation}{section}

\title{The Development of a Wax Layer on the Interior Wall of a Circular Pipe Transporting Heated Oil - The Effects of Temperature Dependent Wax Conductivity}

\usepackage{authblk}

\author{
S. L. Mason \\ 
\href{mailto:SLM585@student.bham.ac.uk}{\texttt{SLM585@student.bham.ac.uk}} 
\and  
J. C. Meyer \\ 
\href{mailto:J.C.Meyer@bham.ac.uk}{\texttt{J.C.Meyer@bham.ac.uk}} 
\and 
D. J. Needham \\
\href{mailto:D.J.Needham@bham.ac.uk}{\texttt{D.J.Needham@bham.ac.uk}} 
}
\affil{School of Mathematics, University of Birmingham, Birmingham, B15 2TT, UK}

\date{21$^{\text{st}}$ April 2021}

\begin{document}

\maketitle

\begin{abstract}
In this paper we develop and significantly extend the thermal phase change model, introduced in \cite{Needham}, describing the process of paraffinic wax layer formation on the interior wall of a circular pipe transporting heated oil, when subject to external cooling. 
In particular we allow for the natural dependence of the solidifying paraffinic wax conductivity on local temperature. 
We are able to develop a complete theory, and provide efficient numerical computations, for this extended model. 
Comparison with recent experimental observations is made, and this, together with recent reviews of the physical mechanisms associated with wax layer formation, provide significant support for the thermal model considered here.
\\ \\
\textbf{keywords - } heated oil pipeline, wax layers, generalised Stefan problem, quasi-linear parabolic PDE, asymptotic limit
\\ \\
\textbf{MSC2020 - }76T99, 80A22, 80M35, 80M20

\end{abstract}


\section{Introduction}\label{Intro}

The transport of oil in long subsea pipelines occurs in many oil-producing fields.
In a recent paper \cite{Needham}, one of the key challenges arising in subsea development of oil fields, namely the formation of paraffinic wax deposits on the inside of the pipe wall, was considered. 
This wax deposit happens when the temperature of the pipe wall falls below the wax solidification temperature, generally in the range $35-40^\circ C$.
Deposition of wax on the pipe wall has significant operational consequences since the inner diameter of the pipe will decrease thereby reducing the transport capacity of the pipe for a given driving pressure.
A more detailed discussion of this phenomena is given in \cite{Needham}. 
In particular \cite{Needham} was concerned with the introduction of a thermal phase change model (first proposed by Schulkes in \cite{Schulkes}) to capture the fundamental mechanism in the wax layer deposition process accurately. 
The associated mathematical model was both fundamental and analysed in detail in \cite{Needham}. 
The outcomes were encouraging, in that a number of key observations in the wax layer deposition process, which were at odds with the historic material diffusion mechanism, were now fully accounted for in the thermal phase change model. 
This model has had further recent support in the independent works  in \cite{LAAT1} and \cite{M1}\footnote{Additionally see \href{https://www.linkedin.com/company/tupdp/?feedView}{TUPDP}.}.

In \cite{Needham}, a thermal approach towards modelling the deposition of paraffinic wax on the inside of the pipe wall was proposed.
The basic assumption introduced in \cite{Needham} is that the principal mechanism leading to wax deposition is a thermal phase change process.
This led to a simple model based on the fundamental balances of heat conduction from the heated oil, through the evolving solid wax layer and pipe wall, and finally into the cooling fluid surrounding the pipe. 
The mathematical formulation of this model gave rise to a free boundary problem (referred to as [IBVP] in \cite{Needham}) of generalised Stefan type. 
This fundamental problem was analysed in considerable detail in \cite{Needham}. 
A number of key salient features arising from the model were identified (see \cite[section 8, p.119-122]{Needham}) with the intention of qualifying and quantifying the basis of the thermal model with detailed experiments performed by Hoffman \& Amundsen \cite{HA1} and Halstensen et al \cite{HAAH1} amongst others.
The experiments of Hoffman \& Amundsen \cite{HA1} show that with a constant flow rate, the wax layer reaches an equilibrium height after sufficient time with this time decreasing as the oil temperature increases. 
This basic feature is not predicted by the molecular diffusion model that is widely applied in the oil and gas industry \cite{SVSN1} \cite{SKSN2} and a ``shear-stripping" mechanism has to be introduced to match experimental observation with model predictions.
A modelling approach based on the thermal phase change mechanism as outlined in \cite{Needham} shows that some fundamental experimental observations can be explained without the need to include rather speculative physical mechanisms.
Very recently, the extensive review by Mehrotra et al \cite{MEH-SK1} has given a thorough consideration of experimental evidence, which provides significant and substantial support for the thermal phase change mechanism introduced in \cite{Needham}, as the principal and key mechanism in the process of wax deposition on the interior wall of pipes transporting heated-oil. 
Most recently, a further review has been provided by Van der Geest et al \cite{vdG1}, which again supplies detailed and critical experimental support for the thermal phase change mechanism.
In addition to these reviews, a recent thesis has addressed this issue experimentally, Mahir \cite{M1}, and has also provided significant support for the thermal mechanism proposed in \cite{Needham}\footnote{An interesting video of the experimental formation of a wax layer can be found at \href{https://www.tupdp.org/?trk=organization-update_share-update_update-text}{TUPDP}, which also provides qualitative visual support for the wax thermal phase change mechanism.}. 
With the very recent emergence of this significant and compelling support for the basic thermal model introduced in \cite{Needham}, the purpose of the present paper is to investigate how the mathematical model in \cite{Needham} can be further developed to accommodate more detailed features associated with wax phase change. 
Specifically, in this paper, we develop the model in \cite{Needham} to account for the dependence of the thermal conductivity of the paraffinic wax solid on local temperature. 
This effect can be significant in many paraffinic wax materials when in solid state and is therefore an effect which should be investigated in the modelling process. 
With this inclusion in the mathematical model developed in \cite{Needham}, the significant change is that the associated free boundary problem [IBVP] becomes nonlinear.
In particular, the partial differential equation, rather than being the associated linear heat conduction equation, becomes a nonlinear (in fact uniformly quasi-linear) strictly parabolic equation, with, in addition, the mixed Robin-type linear boundary condition on the solid interior pipe wall, and the latent heat boundary condition on the free solid wax layer surface, both adopting a generalised nonlinear form. 

\noindent The principal aim of the current paper is to fully investigate the effects of the inclusion of the temperature dependent solid wax thermal conductivity in the model developed in \cite{Needham}. In Section 2, we review and extend the thermal phase change model and it's mathematical formulation. This is followed in Section 3 by an extensive study of associated qualitative results for [IBVP]. 
These structural results are then complemented by consideration of the nature of the solution to [IBVP] as $t\to 0^+$ and $t\to\infty$. 
In Section 5 we devote attention to developing a complete and tractable theory for [IBVP] when the parameter $\varepsilon$ is small, a case which often pertains in physical applications. 
In Section 6 we consider numerical solutions to [IBVP] for comparison with the theory developed earlier, whilst Section 7 gives a qualitative comparison with experiments presented in \cite{HA1}.
Finally we end with a discussion in Section 8.

\section{The Model}\label{Model}
\noindent 
In this section, following \cite{Needham}, we formulate the thermal phase change model which was discussed in section 1.
The heated oil is in uniform flow through a long, straight section of pipe, with circular cross section, and internal radius $R$.
We first observe that when the wax layer thickness on the interior pipe wall $h$ is very much smaller than the pipe radius $R$ ($h \ll R$), which is often the case in applications, then we can reduce the problem to a planar geometry.
Following \cite{Needham} we restrict attention to a long section of the pipe which is remote from entry and exit effects. 
Under these circumstances, we may consider all properties in the model to be independent of axial distance along the pipe.
Thus all dependent variables in the model are functions of $x$ and $t$ alone, where $x$ measures normal distance from the inner pipe wall towards the pipe axis, and $t$ is time.
The situation is illustrated in Figure \ref{fig:WaxLayerDiagramPipeAxis}.
The pipe is surrounded by an aligned circular coolant jacket.  
The fluid in the coolant jacket is maintained at constant temperature $T_c$, with the coolant jacket having width $d_c$.
The thickness of the pipe wall is $d_p$, with the outer pipe wall located at $x=-d_p$ and the inner pipe wall located at $x=0$.
The temperature within the pipe wall is denoted by $T_p$.
The solid wax layer is initiated, when $t=0$, at the inner pipe wall $x=0$, with its upper surface at $x=h$.
The temperature within the solid wax layer is denoted by $T$, with $T_h$ being the constant temperature of the wax formation, at $x=h$.
The temperature of the oil flowing in the pipe, when $h<x\leq R$, is taken as constant, and represented by $T_o$, an approximation which observations confirm to be reasonable in applications \cite{LAAT1}.
Throughout, we consider the situation when,  
\begin{equation} \label{DNE2.1}
T_c < T_h < T_o.
\end{equation}

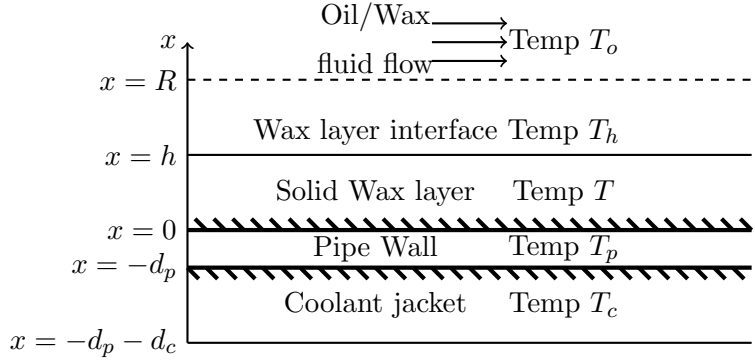
\begin{figure}[]
    \centering
	\begin{tikzpicture}[thick, scale=0.5]
	\draw[-] 	(0,0) -- (15,0);
	\draw[->] 	(0,0) -- (0,8);
	\node[left] at (0,8) {$x$};
	\draw[-] 	(0,5) -- (15,5);
	\draw[-, dashed] 	(0,7) -- (15,7);	
	\node[left] at (0,5) {$x = h$};
	\node[left] at (0,7) {$x = R$};
	\draw[-, ultra thick] 	(0,2) -- (15,2);
	\draw[-, ultra thick] 	(0,3) -- (15,3);
	\foreach \x in {0,...,29} 
	{\draw[-, ultra thick] ( \x * .5  , 2) -- ( \x * .5 + .3, 2 - .3);}
	\foreach \x in {1,...,30} 
	{\draw[-, ultra thick] ( \x * .5  , 3) -- ( \x * .5 - .3, 3 + .3);}
	\node[left] at (0,2) {$x = - d_p$};
	\node[left] at (0,3) {$x = 0$};
	\node[left] at (0,0) {$x = -d_p - d_c$};
	\node at (5,1) {Coolant jacket};
	\node at (10,1) {Temp $T_c$};
	\node at (5,2.5) {Pipe Wall};
	\node at (10,2.5) {Temp $T_p$};
	\node at (5,4) {Solid Wax layer};
	\node at (10,4) {Temp $T$};
	\node[above] at (5,5) {Wax layer interface};
	\node[above] at (10,5) {Temp $T_h$};	
	\node[above] at (5,8) {Oil/Wax};
	\node[below] at (5,8) {fluid flow};
	\node at (10,8) {Temp $T_o$};
	\draw[->] (6.5,8) -- (8.5,8);
	\draw[->] (6.5,8.5) -- (8.5,8.5);
	\draw[->] (6.5,7.5) -- (8.5,7.5);		
\end{tikzpicture}
    \caption{Schematic diagram of the physical problem.}
    \label{fig:WaxLayerDiagramPipeAxis}

\end{figure}
\noindent
We now consider the temperature field in the pipe wall. 
An application of Fourier's law gives
\begin{equation}\label{S2e1}
T_{pt}=\frac{k_p}{c_p\rho_p}T_{pxx}, \quad -d_p < x<0,\ t>0,
\end{equation}
where $k_p$, $c_p$ and $\rho_p$ are the conductivity, specific heat capacity and density of the pipe wall respectively. 
With $t_s$ (as given in \eqref{Scales} and \eqref{DNE2*}) being the time scale associated with wax layer formation, the thickness of the pipe wall is thin, and such that
\begin{equation*}
d_p^2 \ll \frac{\left(k_p t_s \right)}{\left( c_p  \rho_p \right)}.
\end{equation*}
Consequently, the temperature in the pipe wall is in a quasi-steady state, and equation \eqref{S2e1} may be approximated by
\begin{equation}\label{PPDE}
T_{pxx} = 0; \quad -d_p<x<0, \quad t \geq 0,
\end{equation}
subject to the following boundary conditions,
\begin{align}
\label{PBC1}
& k_pT_{px}=\frac{k_c{N\!u}_c}{d_c}(T_p-T_c);&\quad x&=-d_p,\quad t \geq 0, \\
\label{PBC2}
&k_pT_{px}=k_w(T)T_x;&\quad x&=0,\quad t \geq 0, \\
\label{PBC3}
&T_p=T;&\quad x&=0,\quad t \geq 0.
\end{align}
Here $k_c$ is the coolant conductivity, $N\!u_c$ is the Nusselt number for the coolant flow and $k_w(T)$ is the temperature dependent solid wax conductivity. 
Condition \eqref{PBC1} represents continuity of heat flux across the exterior wall of the pipe, which is in contact with the coolant jacket, whilst condition \eqref{PBC2} represents continuity of heat flux from the lower boundary of the wax layer into the interior pipe wall, and condition \eqref{PBC3} represents continuity of temperature at the lower boundary of the wax layer and the inner pipe wall. 
In allowing the solid wax conductivity to be temperature dependent, we write 
$$
k_w(T)={k_w^h}D\left(\frac{T-T_c}{T_h-T_c}\right) , 
$$ 
for $-\infty < T < \infty$, where ${k_w^h}$ is the conductivity of the solid wax at the temperature of wax solidification $T=T_h$, and $D:\mathbb{R} \rightarrow \mathbb{R}$ is the dimensionless solid wax conductivity, representing the variation in solid wax conductivity with, 
$$
u=\frac{\left(T-T_c \right)}{\left(T_h - T_c \right)}.
$$ 
In general $D:\mathbb{R}\to\mathbb{R}$ will be taken as smooth, bounded and bounded above zero. 
Thus, throughout, we will consider $D:\mathbb{R}\to\mathbb{R}$ to satisfy the conditions:

\begin{enumerate}
\item [(D1)] $D(1)=1,$
\item [(D2)] $D_m \leq D(u) \leq D_M$ for all $u\in \mathbb{R}$, with $D_m$ and $D_M$ being positive constants,
\item [(D3)] $D \in C^3 ( \mathbb{R} ).$
\end{enumerate}
Now, the solution to \eqref{PPDE}, \eqref{PBC1} and \eqref{PBC3} is given by,
\begin{equation*}
T_p(x,t)=\frac{k_cN\!u_c\left(T(0,t)-T_c\right)}{\left( k_pd_c+k_cN\!u_cd_p \right)}x+T(0,t);\quad -d_p\leq x \leq 0,\quad t \geq 0, 
\end{equation*}
after which boundary condition \eqref{PBC2} then requires,
\begin{equation*}
    k_w(T)T_x=\frac{k_pk_cN\!u_c}{\left(k_pd_c+k_cN\!u_cd_p \right)}(T-T_c);\quad x=0,\quad t > 0,
\end{equation*}
which is a condition on $T$ at $x=0$, the lower boundary of the solid wax layer at the inner pipe wall.
We are now in a position to consider the temperature field $T$ in the solid wax layer. 
Fourier's law requires that,
\begin{equation}\label{WPDE}
\rho_{w}c_{w}T_{t}=(k_w(T)T_x)_x;\quad 0<x<h(t),\quad t>0,
\end{equation}
subject to the boundary conditions,
\begin{align}
\label{WBC1}
& k_w(T)T_x =\frac{k_pk_cN\!u_c}{\left( k_pd_c+k_cN\!u_cd_p \right)}(T-T_c);& x &=0,\quad t > 0, \\
\label{WBC2}
& T =T_h;& x& =h(t), \quad t>0, \\
\label{WBC3}
& \rho_wH_wh_t  =k_w(T)T_x-\frac{k_oN\!u}{R}(T_o-T_h);& x &=h(t),\quad t>0,
\end{align}
where $\rho_w$, $c_w$ and $H_w$ are the density, specific heat capacity and latent heat of the solid wax respectively, whilst $k_o$ and $N\!u$ are the conductivity of oil and the Nusselt number for the oil flow respectively. 
The conditions \eqref{WBC2} and \eqref{WBC3} express that the outer surface of the solid wax layer must be at the wax solidification temperature, and that the difference in heat \textit{flux} across this interface balances the latent heat required for solid wax formation. 
For convenience we non-dimensionalise the free boundary problem for $T(x,t)$ and $h(t)$ given by \eqref{WPDE} - \eqref{WBC3}. 
We introduce the dimensionless variables,
\begin{equation}\label{timescales}
u=\frac{T-T_c}{T_s},\quad x'=\frac{x}{x_s},\quad t'=\frac{t}{t_s},\quad h'=\frac{h}{x_s}.
\end{equation}
with the scales $T_s$, $x_s$ and $t_s$ chosen as,
\begin{equation}\label{Scales}
T_s=T_h-T_c, \quad x_s=\frac{R{k_w^h}\left(T_h-T_c\right)}{k_oN\!u\left(T_o-T_h\right)},\quad t_s=\frac{R^2\rho_wH_w{k_w^h}\left(T_h-T_c\right)}{{k_o}^2{N\!u}^2\left(T_o-T_h\right)^2}.
\end{equation}
On substituting \eqref{timescales} and \eqref{Scales} into \eqref{WPDE} - \eqref{WBC3} we obtain the non-dimensional form of the free boundary problem as,
\begin{align}
\label{NDPDE}
& {\varepsilon}u_t=(D(u)u_x)_x; && 0<x<h(t), \quad t>0,\\
\label{NDBC1}
& D(u)u_x=ku; && x=0, \quad t>0, \\
\label{NDBC2}
& u=1; && x=h(t), \quad t>0, \\
\label{NDBC3}
& h_t=D(u)u_x-1; && x=h(t), \quad t>0,
\end{align}
where the two dimensionless parameters $\varepsilon$ and $k$ are given by,
\begin{equation}\label{epsilonandk}
\varepsilon=\frac{c_w\left(T_h-T_c\right)}{H_w} \text{,\qquad} k=\frac{k_pk_cN\!u_cR\left(T_h-T_c\right)}{k_oN\!u\left(k_pd_c+k_cN\!u_cd_p\right)\left(T_o-T_h\right)}.
\end{equation}
Here primes have been dropped for ease of notation. 
The parameter $\varepsilon$ measures the ratio of the time scale for heat conduction in the wax layer to the time scale for wax layer growth, whilst the parameter $k$ measures the ratio of heat extracted from the wax layer, by cooling, to the heat the wax layer gains from the oil. 
This is a free boundary problem for $u(x,t)$ with $0 \leq x \leq h(t)$, $t>0$ and $h(t)\geq 0$.
It is worth noting that typical values, estimated in Schulkes \cite{Schulkes} and Kaye and Laby \cite{KL1}, for the scales in (\ref{Scales}), are
\begin{equation} \label{DNE2*}
T_s \sim 10 \degree C, \quad x_s \sim 0.16\text{m} \text{, \quad} t_s \sim 2 \times 10^6\text{s} \sim 20 \text{ days}, 
\end{equation} 
with the dimensionless parameters,
 \begin{equation*}
 \varepsilon \sim O \left(10^{-1}\right) \text{, \quad} k \sim O \left(\frac{2N\!u_c}{1+10^{-3}N\!u_c} \right).
\end{equation*}

\section{The Free Boundary Problem [IBVP]} \label{FBP}
\noindent 
The free boundary problem associated with the mathematical model introduced in section \ref{Model} (\eqref{NDPDE} - \eqref{NDBC3}) may be written fully as,
\begin{align}
\label{2NDPDE}
&{\varepsilon}u_t=(D(u)u_x)_x;&& 0<x<h(t), \quad t>0, \\
\label{2NDBC1}
& D(u)u_x=ku; && x=0, \quad t>0, \\
\label{2NDBC2}
& u=1; && x=h(t), \quad t>0, \\
\label{2NDBC3}
& h_t=D(u)u_x-1; && x=h(t), \quad t>0, \\
\label{2NDBC4}
& u \rightarrow 1 \text{ as } t \rightarrow0^+ \text{ uniformly for }0\leq x \leq h(t), \\
\label{2NDBC5}
& h \rightarrow 0 \text{ as } t \rightarrow 0^+. 
\end{align}
The problem, \eqref{2NDPDE} - \eqref{2NDBC5}, will be referred to as [IBVP]. 
For any $T>0$, the following subsets of ${\mathbb{R}}^2$ are also introduced, namely,
\begin{align*}
&D_T=\{\left(x,t\right)\in {\mathbb{R}}^2:0<x<h(t), 0< t \leq T \},\\
&\partial{D_L^T}=\{\left(x,t\right)\in {\mathbb{R}}^2:x=0, 0< t \leq T \},\\
&\partial{D_R^T}=\{\left(x,t\right)\in {\mathbb{R}}^2:x=h(t), 0< t \leq T \},\\
&\partial{D_T}=\partial{D_L^T} \cup \partial{D_R^T},
\end{align*}
with closures denoted by $\overline{D}_T$, ${\partial{\overline{D}_L^T}}$, ${\partial{\overline{D}_R^T}}$ and ${\partial{\overline{D}_T}}$.
A solution to [IBVP] will be considered as classical, with the following regularity requirements,
\begin{enumerate}[(\roman{enumi})]
\item[(R1)] $h\hspace{-0.5mm}:\hspace{-0.5mm}\left[0,\infty\right)\rightarrow\mathbb{R}$ is continuous and the derivative $h_t$ exists and is continuous on $\left(0,\infty\right)$, with $h$ and $h_t$ non-negative on $(0,\infty).$
\item[(R2)] $u\hspace{-0.5mm}:\hspace{-0.5mm}\overline{D}_{\infty}\rightarrow\mathbb{R}$ is continuous and $u_x,u_t$ both exist and are continuous on $\overline{D}_{\infty}$ and $u_{xx}$ exists and is continuous on ${D_\infty}$.
\end{enumerate}
A reformulation of [IBVP] (with (R1) and (R2)) in terms of coupled integral equations is given by Friedman \cite{Friedman1992} and used by Schatz \cite{Schatz1969} and Cannon and Hill \cite{Cannon1967} to study the regularity of solutions to [IBVP]. 
It is established by Cannon and Hill \cite{Cannon1967} that with $u:\overline{D}_{\infty}\rightarrow\mathbb{R}$ and $h:[0,\infty )\to\mathbb{R}$ being a solution to [IBVP], (with (R1) and (R2)) then, in fact, 
\begin{equation}\label{uinC}
u \in C^{3} (\overline{D}_{\infty}) \quad \text{ and }\quad h \in C^{1}([0, \infty )),
\end{equation}
which requires (D3), in particular.
A consequence of (R1) and \eqref{uinC} is also, 
\begin{equation} \label{4.*}
h(t) \geq 0 \quad \text{ and} \quad h_t(t) \geq 0 \quad \forall \  t\in [0,\infty).
\end{equation}
Before proceeding to analyse [IBVP] further, we first consider steady state solutions associated with [IBVP].

\subsection{Steady State Solutions to [IBVP]} \label{steadystate}
\noindent
A steady state solution to [IBVP] is a solution to [IBVP] which is independent of $t$. 
Thus, $u_s: [0,h_s]\to\mathbb{R}$ and $h_s>0$ is a steady solution to [IBVP] whenever $u_s\in C^1([0,h_s])\cap C^2 ((0,h_s)),$ and satisfies the boundary value problem,
\begin{align}
\label{SSPDE}
& (D(u_s)u_{sx})_x=0; && 0<x<h_s, \\
\label{SSBC1}
& D(u_s)u_{sx}=ku_s; && x=0, \\
\label{SSBC2}
& u_s=1; && x=h_s, \\
\label{SSBC3}
& D(u_s)u_{sx}=1; && x=h_s. 
\end{align}
It is now convenient to introduce $F:\mathbb{R}\to\mathbb{R}$ given by,
\begin{equation}\label{eq:s4definitionF}
F(X)=\int_{0}^{X}D(\lambda)d\lambda \quad \forall \  X\in\mathbb{R},
\end{equation}
where $F(0)=0$ and,
\begin{equation}\label{eq:s4definitionF2}
F'(X)=D(X)>0 \quad \forall \  X\in\mathbb{R}.
\end{equation}
Observe, via (D2) and (D3), that $F(X)$ is strictly increasing with $X$, and $F \in C^4(\mathbb{R})$.
Therefore the inverse $F^{-1}:\mathbb{R}\to\mathbb{R}$ exists, with $F^{-1}\in C^4 (\mathbb{R})$, via \eqref{eq:s4definitionF2}. 
The boundary value problem \eqref{SSPDE} - \eqref{SSBC3} can now be written as,
\begin{align}
\label{2SSPDE}
& \left(F(u_s)\right)_{xx}=0; && 0<x<h_s, \\
\label{2SSBC1}
& (F(u_s))_x=ku_s; && x=0, \\
\label{2SSBC2}
& u_s=1; && x=h_s, \\
\label{2SSBC3}
& (F(u_s))_x=1; && x=h_s. 
\end{align}
An integration of \eqref{2SSPDE} gives, 
\begin{equation}\label{eq6}
F(u_s(x))=Ax+B; \quad 0\leq x\leq h_s,
\end{equation}
where $A,B\in\mathbb{R}$ are constants. 
Applying \eqref{2SSBC1} and \eqref{2SSBC3} we obtain, 
\begin{equation} \label{DNE4.18}
A=1 \quad \text{ and } \quad B=F\left(k^{-1}\right) , 
\end{equation}
after which, \eqref{eq6} becomes,
\begin{equation}\label{S4:eq2}
F(u_s(x))=x+F\left( k^{-1} \right);\quad 0\leq x\leq h_s.
\end{equation}
Finally, applying \eqref{2SSBC2} and rearranging, we require,
\begin{equation}\label{S4:eq1}
h_s=F(1)-F\left( k^{-1} \right) .
\end{equation}
In a steady state we must have $h_s>0$. 
Therefore, it follows from \eqref{S4:eq2} and \eqref{S4:eq1} that we have established:

\begin{proposition} \label{DNEp1}
A steady state solution to [IBVP] exists, and is unique, if and only if, $k>1$.
The steady state solution is given by \eqref{S4:eq2} with \eqref{S4:eq1}.
\end{proposition}
\noindent
It is worth noting that \eqref{S4:eq1} may be written as,
\begin{equation}\label{S4:eq3}
h_s=\left(1-k^{-1} \right)\left< D\left(\left[ k^{-1},1\right]\right) \right>,
\end{equation}
where
\begin{equation}\label{S4:eq3.5}
\left< D\left(\left[ k^{-1} ,1\right]\right) \right>=\left(1- k^{-1} \right)^{-1}\int_{ k^{-1}}^{1}D(\lambda)d\lambda,
\end{equation}
is the mean value of $D(\lambda)$ over the interval $\lambda \in \left[k^{-1} ,1\right]$.
Also, we may rewrite \eqref{S4:eq2} explicitly as 
\begin{equation}\label{S4:eq4}
u_s(x)=F^{-1}\left(x+F\left( k^{-1} \right)\right); \quad 0 \leq x \leq h_s.
\end{equation}
An examination of $h_s=h_s(k)$ in \eqref{S4:eq3} establishes the following properties, 
\renewcommand{\labelenumi}{\roman{enumi}}
 \begin{enumerate}[(\roman{enumi})]
   \item $h_s \in C^4([1,\infty))$ and is strictly monotone increasing,
   \item $h_s'(k)=D\left( k^{-1} \right) k^{-2} \quad \forall \  k \in [1,\infty),$   
   \item $h_s(k)=\left(k-1\right) + O\left(\left(k-1\right)^2\right)$ as $k \rightarrow 1^+,$
   \item $h_s(k)=\left<D([0,1])\right> - D(0) k^{-1} + O\left( k^{-2}\right)$ as $k\rightarrow\infty.$\\
 \end{enumerate}
We give a qualitative sketch of $h_s(k)$ against $k$ in Figure \ref{fig:sketch for h_s}. 
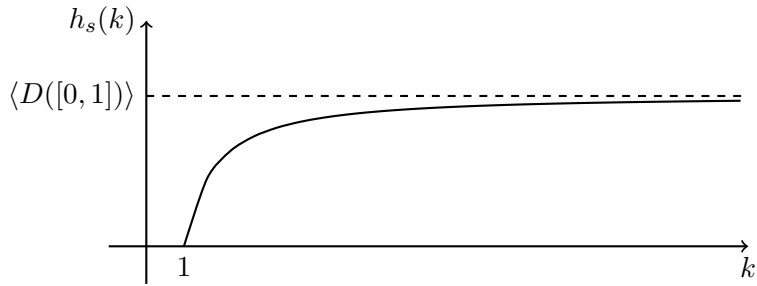
\begin{figure}[h!]
    \centering
    \centering
	\begin{tikzpicture}[thick, scale=0.5]
	\draw[->] (-1,0) -- (16,0);
	\draw[->] 	(0,-1) -- (0,6);
	\node[left] at (0,6) {$h_s(k)$};
	\node[below] at (16,0) {$k$}; 
	\draw[scale=1, domain=1:15.8, smooth, variable=\x, black] plot ({\x},{4 - 4/(\x^(5/4))}); 
	\draw[-, dashed] 	(0,4) -- (16,4);	
	\node[left] at (0,4) {$\langle D([0,1]) \rangle$};
	\node[below] at (1,0) {$1$}; 
\end{tikzpicture}
    \caption{Sketch of $h_s(k)$ against $k$.}
    \label{fig:sketch for h_s}
\end{figure}
Next, an examination of \eqref{S4:eq4} establishes the following properties of $u_s:[0, h_s]\rightarrow \mathbb{R},$ 
\renewcommand{\labelenumi}{\roman{enumi}}
 \begin{enumerate}[(\roman{enumi})]
 \item $u_s \in C^4(\left[0,h_s\right])$ and is monotone increasing,
 \item $u_s(0)= k^{-1}$ and $u_s(h_s)=1,$
 \item $u_s'(x)=\dfrac{1}{D(u_s(x))}>0 \quad \forall \  x \in [0,h_s],$
 \item $u_s''(x)=-\dfrac{D' (u_s(x))}{{D(u_s(x))}^3} \quad \forall \  x \in [0,h_s]$.
\end{enumerate}
It is worth noting from (iv) that inflection points will occur in the graph of $u_s(x)$ for $x\in [0,h_s]$ if and only if there exists values $u_i \in \left[ k^{-1} ,1\right]$ such that $D'(u_i)=0$. 
For illustrative purposes, steady state solutions associated with those specific solid wax conductivities $D:\mathbb{R} \to \mathbb{R}$ detailed in Section \ref{numerics} are given in Figures \ref{fig:3}. 
We note that for $k=2$ the steady state solutions have no inflection points, whereas for $k>2$ the steady state solutions (except when $c=0$) have a single inflection point. 
We now return to [IBVP].
\begin{figure}[h!]
\begin{subfigure}{.5\textwidth}
  \centering
  \includegraphics[width=\linewidth]{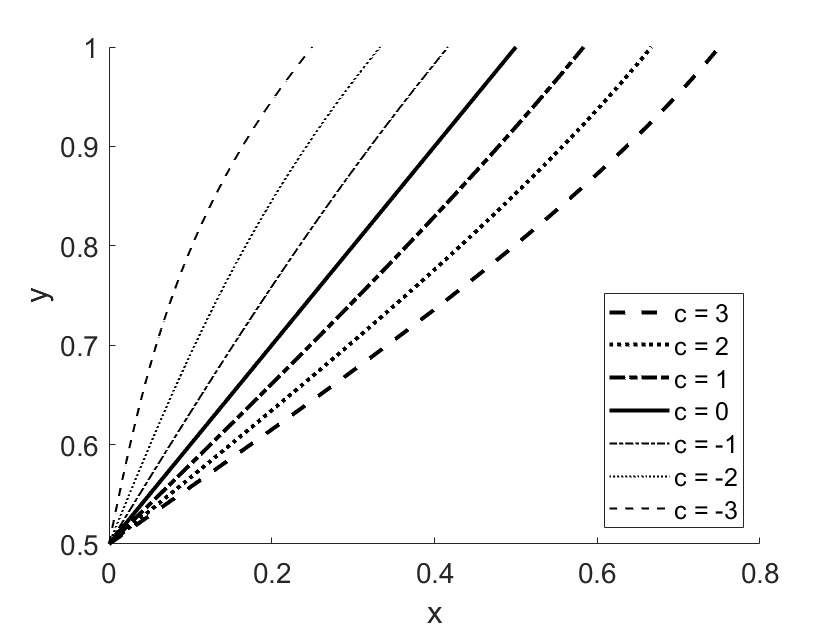}
  \caption{}
  \label{fig:3a}
\end{subfigure}
\begin{subfigure}{.5\textwidth}
  \centering
  \includegraphics[width=\linewidth]{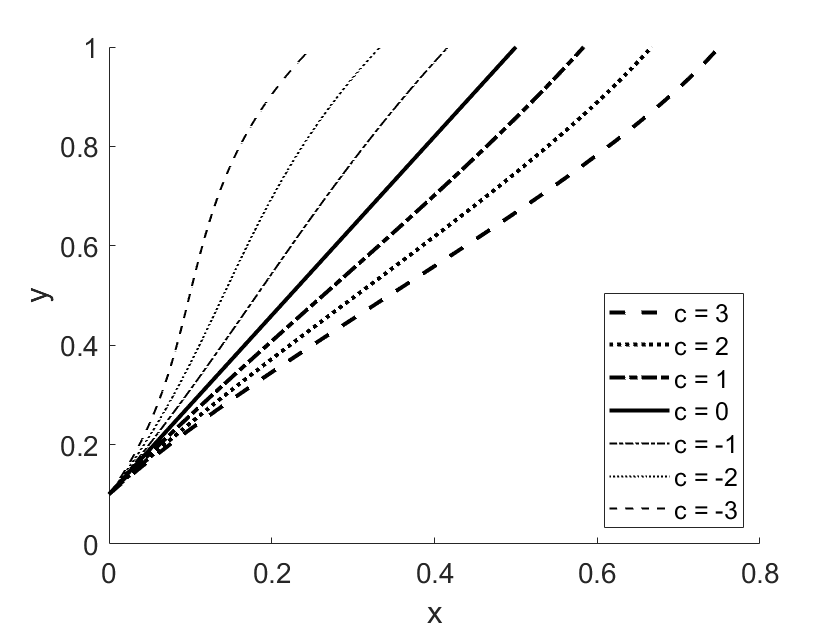}
  \caption{}
  \label{fig:3b}
\end{subfigure}
\caption{
Plots of $y=u_s(x)$ for $D(u)=1+cu(1-u)$ with $k=2$ and $10$ in (a) and (b) respectively. 
}
\label{fig:3}
\end{figure}

\subsection{Qualitative Theory for [IBVP]}
\noindent
In this subsection we determine the principal structure and qualitative properties of the solution to the free boundary problem [IBVP]. 
To begin with, we have, 
\begin{proposition}\label{prop1}
Let  $u:\overline{D}_{\infty}\rightarrow\mathbb{R}$ be a solution to [IBVP]. 
Then
$$
0 \leq u(x,t) \leq 1 \quad \forall \  \left(x,t\right)\in \overline{D}_{\infty}.
$$
\end{proposition}

\begin{proof}
Consider $u:\overline{D}_{\infty}\rightarrow\mathbb{R}$ on the compact region $\overline{D}_T$ (for any $T>0$). 
Then $u$ is continuous on $\overline{D}_T$ since $u$ is continuous on $\overline{D}_{\infty}$ (via (R2)). 
Also, we have that $u_x$, $u_t$ and $u_{xx}$ all exist and are continuous on $D_{\infty}$ and so are continuous on $D_T$ (via (R2)). 
Also, we have from \eqref{2NDPDE} that
\begin{equation*}\label{2NDPDE1}
{\varepsilon}u_t=(D(u)u_x)_x \quad \text{ on } D_{\infty},
\end{equation*}
and so,
\begin{equation}\label{s4.2:1}
u_t-\varepsilon^{-1}D'(u){u_x}^2-\varepsilon^{-1}D(u)u_{xx}=0\quad \text{ on } D_{\infty}.
\end{equation}
Now for each $\left(x,t\right)\in D_T$, set,
\begin{equation} \label{JME2}
a(x,t)=\varepsilon^{-1}D(u)\geq \varepsilon^{-1}D_m>0,
\end{equation}
via (D2) and,
\begin{equation}\label{s4.2:2}
b(x,t)=\varepsilon^{-1}D'(u)u_x.
\end{equation}
Now suppose that $u$ is not non-negative on $\overline{D}_T$. 
Then using \eqref{s4.2:1}-\eqref{s4.2:2}, we may apply the Weak Parabolic Minimum Principle \cite{WW} to conclude that there exists a point $\left(x^*,t^*\right)\in {{\partial{\overline{D}_T}}}$ such that, 
\begin{equation}\label{eq8}
u(x^*,t^*)= \inf_{(x,t)\in\overline{D}_T} u(x,t) <0.
\end{equation}
Since $\left(x^*,t^*\right)\in {{\partial{\overline{D}_T}}}={{\partial{\overline{D}^T_L}}} \cup {{\partial{\overline{D}^T_R}}}$ we must have $\left(x^*,t^*\right)=\left(0,t^*\right)\in {{\partial{\overline{D}^T_L}}}$ otherwise $u(x^*,t^*)=1$. Thus, at $\left(0,t^*\right)$ we have, 
\begin{equation*}
D(u(0,t^*))u_x(0,t^*)=ku(0,t^*)<0,
\end{equation*} 
via condition \eqref{2NDBC1} and \eqref{eq8}. 
Then, via (D2), we conclude that, 
\begin{equation}\label{eq9}
u_x(0,t^*)<0.
\end{equation}
However, since $u$ achieves its infimum at $\left(0,t^*\right)$ then $u_x(0,t^*) \geq 0,$ which contradicts \eqref{eq9}. 
Thus, we conclude that $u$ must be non-negative on $\overline{D}_T$. 
This holds for any $T>0$ and so $u(x,t) \geq 0$ for all $\left(x,t\right)\in \overline{D}_{\infty}$.
Next it follows from \eqref{2NDBC2} that,
\begin{equation}\label{eq10}
\underset{\left(x,t\right)\in\overline{D}_T}{\sup}\hspace{-2mm}u\left(x,t\right) \geq 1.
\end{equation}
However, via \eqref{s4.2:1}-\eqref{s4.2:2}, it follows from the Weak Parabolic Maximum Principle \cite{WW} that there exists a point $\left(x^*,t^*\right) \in {\partial{\overline{D}_T}}$ such that,
\begin{equation}\label{eq11}
u(x^*,t^*)= \sup_{(x,t)\in\overline{D}_T}u(x,t) \geq 1,
\end{equation}
via \eqref{eq10}. 
Now suppose that $\left(x^*,t^*\right) \in {\partial{\overline{D}^T_L}},$ so that $\left(x^*,t^*\right)=\left(0,t^*\right)$, and from \eqref{2NDBC1} and \eqref{eq11} we have that, 
\begin{equation*}
D(u(0,t^*))u_x(0,t^*)=ku(0,t^*) \geq k>0,
\end{equation*} 
and so,
\begin{equation}\label{eq12}
u_x(0,t^*)>0,
\end{equation}
via (D2). 
However, since $u$ achieves its supremum at $\left(0,t^*\right)$ then $u_x(0,t^*) \leq 0$, which contradicts \eqref{eq12}. 
Hence we conclude that $\left(x^*,t^*\right) \in {\partial{\overline{D}^T_R}}$, and so, via \eqref{2NDBC2}, 
\begin{equation*}
u(x^*,t^*)= \sup_{(x,t)\in\bar{D}_T} u(x,t) =1 .
\end{equation*}
Therefore, $u \leq 1$ on $\overline{D}_T$. 
Since this holds for any $T>0$, we conclude that $u(x,t) \leq 1$ for all $\left(x,t\right) \in \overline{D}_{\infty}$. 
Consequently, we have, 
\begin{equation*}
0 \leq u(x,t) \leq 1 \quad \forall \  \left(x,t\right)\in \overline{D}_{\infty},
\end{equation*}
as required. 
\end{proof}
\noindent
We next refine these inequalities in,
\begin{corollary}\label{coro2}
Let  $u:\overline{D}_{\infty}\rightarrow\mathbb{R}$ be a solution to [IBVP]. 
Then, 
$$
0 < u(x,t) < 1 \quad \forall \ \left(x,t\right)\in {D_\infty}.
$$
\end{corollary}

\begin{proof}
From Proposition \ref{prop1} we have, 
\begin{equation*}
0 \leq u(x,t) \leq 1 \quad \forall \ \left(x,t\right)\in \overline{D}_{\infty}.
\end{equation*}
Now suppose there exists a point $\left(x^*,t^*\right) \in D_{\infty}$ such that $u(x^*,t^*)=0$.
Then, 
\begin{equation*}
\inf_{(x,t)\in\overline{D}_\infty}u(x,t)=0
\end{equation*}
and it follows from the Strong Parabolic Minimum Principle \cite{WW} that, 
\begin{equation*}
u(x,t)=0 \quad \forall \ \left(x,t\right) \in \overline{D}_{T^*},
\end{equation*}
where 
$$
T^*=\sup\{t\in \left(0,T\right]:\exists\ x\in\left(0,h(t)\right) \text{ with } u(x,t)=0\} \geq t^* >0.
$$
This contradicts condition \eqref{2NDBC2}. 
Thus, $u(x,t)>0$ for all $\left(x,t\right) \in D_{\infty}$. 
Next suppose that there exists $\left(x^*,t^*\right) \in D_{\infty}$ such that $u(x^*,t^*)=1$.
Then,   
\begin{equation}\label{prop2:eq1}
\underset{\left(x,t\right)\in \overline{D}_{\infty}}{\sup}\hspace{-2mm}u(x,t) = 1,
\end{equation}
and it follows from the Strong Parabolic Maximum Principle \cite{WW} that, 
\begin{equation*}
u(x,t)=1 \quad \forall \  \left(x,t\right) \in \overline{D}_{T^*},
\end{equation*}
where now,
$$
T^*=\sup\{t\in \left(0,T\right]:\exists\ x\in\left(0,h(t)\right)\text{ with } u(x,t)=1\} \geq t^* >0.
$$
Also, via condition \eqref{2NDBC1} with regularity condition (R2), we have from \eqref{prop2:eq1}, 
\begin{equation*}
D(u(0,t))u_x(0,t)=ku(0,t) = k>0 \quad\forall \  t \in \left[0,T^*\right],
\end{equation*}
and so, via (D2),
\begin{equation}\label{eq13}
u_x(0,t)>0 \quad \forall \   t \in \left[0,T^*\right].
\end{equation}
However, since $u(x,t)=1$ for all $(x,t)\in \overline{D}_{T^*}$, then $u_x(0,t)=0$ for all $t \in \left[0,T^*\right]$, which contradicts \eqref{eq13}. 
Thus $u(x,t)<1$ for all $(x,t) \in D_{\infty}$.
Therefore,
\begin{equation*}
0 < u(x,t) < 1 \quad \forall \ \left(x,t\right)\in {D_\infty},
\end{equation*}
as required. 
\end{proof}

\noindent
Next we have,
\begin{proposition} \label{prop3}
Let $u:\overline{D}_{\infty}\rightarrow\mathbb{R}$ be a solution to [IBVP]. 
Then,
\begin{equation*}
u_t(x,t) \leq 0 \quad \forall \ \left(x,t\right)\in \overline{D}_{\infty}.
\end{equation*}
\end{proposition}

\begin{proof}
Firstly, from \eqref{2NDPDE} we have,
\begin{equation}\label{prop3:1}
\varepsilon D(u) u_t = D(u)(D(u) u_x)_x \quad\text{ on } D_{\infty}.
\end{equation}
Next we introduce $f:\overline{D}_{\infty}\rightarrow\mathbb{R}$ defined by,
\begin{equation}\label{prop3:2}
f(x,t)=F(u(x,t)) \quad \forall \ \left(x,t\right)\in \overline{D}_{\infty},
\end{equation}
with $F:\mathbb{R}\to\mathbb{R}$ as introduced in \eqref{eq:s4definitionF}.
It follows from \eqref{eq:s4definitionF}, \eqref{uinC} and (D3) that,
\begin{equation}\label{prop3:3}
f \in C^{3}\hspace{-0.5mm}\left(\overline{D}_{\infty}\right) \hspace{-0.5mm},
\end{equation}
and we have, for $\left(x,t\right) \in \overline{D}_{\infty},$
\begin{align}
\label{prop3:4}
&f_t(x,t)=F'(u(x,t))u_t(x,t)=D(u(x,t))u_t(x,t),\\
\label{prop3:5}
&f_x(x,t)=F'(u(x,t))u_x(x,t)=D(u(x,t))u_x(x,t),\\
\label{prop3:6}
&f_{xx}(x,t)=D'(u(x,t))u_x^2(x,t)+D(u(x,t))u_{xx}(x,t).
\end{align}
Therefore, via \eqref{prop3:4} and \eqref{prop3:5}, we observe that \eqref{prop3:1} becomes,
\begin{equation}\label{prop3:7}
\varepsilon f_t=D(u)f_{xx} \quad \text{ on } D_{\infty}.
\end{equation}
Similarly, \eqref{2NDBC1}-\eqref{2NDBC3} becomes, 
\begin{align}
\label{prop3:8}
& f_x=ku; && x=0, \quad t>0,
\\
\label{prop3:9}
& f=F(1); && x=h(t), \quad t>0, 
\\
\label{prop3:10}
& h_t=f_x-1; && x=h(t), \quad t>0. 
\end{align}
We next introduce $w:\overline{D}_{\infty}\rightarrow\mathbb{R}$ defined by,
\begin{equation}\label{prop3:11}
w(x,t)=f_t(x,t)\quad \forall \ \left(x,t\right)\in \overline{D}_{\infty}.
\end{equation}
From \eqref{prop3:3} it follows that,
\begin{equation}\label{prop3:11.5}
w \in C^{2} \hspace{-0.5mm}\left(\overline{D}_{\infty}\right)\hspace{-0.5mm},
\end{equation} 
whilst from \eqref{prop3:3} and \eqref{prop3:7}, we have that,
\begin{equation*}\label{prop3:12}
\varepsilon f_{tt}=D(u)f_{xxt}+D'(u)u_t f_{xx} \quad \text{ on } D_{\infty},
\end{equation*}
which becomes, via \eqref{prop3:4} and \eqref{prop3:11},
\begin{equation*}\label{prop3:13}
\varepsilon w_t=D(u)w_{xx}+\frac{D'(u)}{D(u)}f_{xx}w \quad \text{ on } D_{\infty}.
\end{equation*}
Hence, we have,
\begin{equation}\label{prop3:14}
w_t-a(x,t)w_{xx}-c(x,t)w=0 \quad \text{ on } D_{\infty},
\end{equation}
where,
\begin{equation}\label{prop3:15}
a(x,t)= \frac{D(u(x,t))}{\varepsilon}>\frac{D_m}{\varepsilon}>0 \quad \forall \ \left(x,t\right)\in {\overline{D}_\infty}, 
\end{equation}
via (D2), and, 
\begin{equation*}\label{prop3:16}
c(x,t)=\frac{D'(u(x,t))}{\varepsilon D(u(x,t))}f_{xx}(x,t) \quad \forall \ \left(x,t\right)\in {\overline{D}_\infty}.
\end{equation*}
It follows from \eqref{uinC}, (D3) and \eqref{prop3:3} that $c(x,t)$ is continuous on $\overline{D}_{\infty}$, and so is continuous and bounded on $\overline{D}_{T}$, for any $T>0$. 
Also, via \eqref{prop3:8} together with \eqref{uinC} and \eqref{prop3:3},
\begin{equation*}\label{prop3:17}
f_{xt}=ku_t;\quad x=0, \quad t>0,
\end{equation*}
and so, via \eqref{prop3:4} and \eqref{prop3:11},
\begin{equation}\label{prop3:18}
w_x=\frac{kw}{D(u)}; \quad x=0, \quad t>0.
\end{equation}
In addition, via \eqref{prop3:3}, \eqref{prop3:9}, regularity condition (R1) and the chain rule we have,
\begin{equation*}\label{prop3:19}
f_t+h_t f_x =0; \quad x=h(t), \quad t>0,
\end{equation*}
and so, via \eqref{prop3:11},
\begin{equation*}\label{prop3:20}
w=-h_t f_x;\quad x=h(t), \quad t>0.
\end{equation*}
However, regularity condition (R1) requires $h_t \geq 0$, and so, via \eqref{prop3:10},
\begin{equation}\label{prop3:21}
f_x \geq 1; \quad x=h(t), \quad t>0.
\end{equation}
Thus, we conclude from regularity condition (R1) and \eqref{prop3:21} that, 
\begin{equation}\label{prop3:22}
w \leq 0;  \quad x=h(t), \quad t>0.
\end{equation}
Next we introduce $v:\overline{D}_{\infty}\rightarrow\mathbb{R}$ such that,
\begin{equation}\label{prop3:23}
w(x,t)=e^{\lambda t}v(x,t) \quad \forall \ \left(x,t\right)\in \overline{D}_{\infty},
\end{equation}
with $\lambda \in \mathbb{R}$ to be chosen. 
From \eqref{prop3:11.5} it follows that, 
\begin{equation}\label{prop3:24}
v \in C^{2} \hspace{-0.5mm}\text{}\left(\overline{D}_{\infty}\right)\hspace{-0.5mm}. 
\end{equation}
Thus, for $\left(x,t\right) \in \overline{D}_{\infty}$ we have,
\begin{align}
\label{prop3:24.5}
&w_x(x,t)=e^{\lambda t}v_x(x,t),
\\
&w_t(x,t)=\lambda e^{\lambda t}v(x,t)+e^{\lambda t}v_t(x,t),
\\
&w_{xx}(x,t)=e^{\lambda t}v_{xx}(x,t).
\end{align}
Then, via \eqref{prop3:23} and \eqref{prop3:24}, with \eqref{prop3:14}, we obtain,
\begin{equation}\label{prop3:25}
v_t-a(x,t)v_{xx}-\left(c(x,t)-\lambda\right)v=0 \quad \text{on } D_{\infty}.
\end{equation} 
Now set $T>0$ and let $C_T>0$ be a bound for $c$ on $\overline{D}_{T}$, so that, 
\begin{equation}\label{prop3:26}
|c(x,t)| \leq C_T \quad \forall \ \left(x,t\right)\in \overline{D}_{T}.
\end{equation}
We next choose
\begin{equation}\label{prop3:27}
\lambda=2C_T >0,
\end{equation}
so that,
\begin{equation}\label{prop3:28}
c(x,t)-\lambda \leq -C_T < 0 \quad \forall \ \left(x,t\right)\in \overline{D}_{T}.
\end{equation}
Suppose that $v$ is not non-positive on $\overline{D}_{T}$. 
Then, 
\begin{equation*}\label{prop3:29}
\underset{\left(x,t\right)\in \overline{D}_{T}}{\sup}\hspace{-2mm}v(x,t) > 0,
\end{equation*}
and, via \eqref{prop3:22} and \eqref{prop3:23}, this cannot be achieved on ${\partial{\overline{D}^T_R}}$. 
Thus, there exists,
\begin{equation}\label{prop3:30}
\left(x^*,t^*\right)\in D_T \cup \partial{D^T_L},
\end{equation}
such that,
\begin{equation}\label{prop3:31}
v(x^*,t^*)= \sup_{(x,t)\in\overline{D}_T}v(x,t)>0 .
\end{equation}
If $\left(x^*,t^*\right)\in D_T$ then $t^*\in\left(0,T\right)$ or $t^*=T$. 
When $t^* \in \left(0,T\right)$ then $\left(x^*,t^*\right)$ is such that,
\begin{align}
\label{prop3:32}
& v_t(x^*,t^*)=v_x(x^*,t^*)=0,
\\
\label{prop3:33}
& v_{xx}(x^*,t^*) \leq 0.
\end{align}
However, from \eqref{prop3:25} with \eqref{prop3:15}, \eqref{prop3:28}, \eqref{prop3:31} and \eqref{prop3:33}, we have,
\begin{equation*}\label{prop3:33.5}
v_t(x^*,t^*)=a(x^*,t^*)v_{xx}(x^*,t^*)+\left(c(x^*,t^*)-\lambda\right)v(x^*,t^*) <0,
\end{equation*}
which contradicts \eqref{prop3:32}. 
Similarly, when $t^*=T$, then $\left(x^*,T\right)$ is such that,
\begin{equation}\label{prop3:34}
v_x(x^*,T)=0, \quad v_{xx}(x^*,T) \leq 0,
\end{equation}
and
\begin{equation}\label{prop3:34.1}
v_t(x^*,T) \geq 0.
\end{equation}
However, via \eqref{prop3:25} with \eqref{prop3:15}, \eqref{prop3:28}, \eqref{prop3:31} and \eqref{prop3:34}, we have,
\begin{equation*}\label{prop3:34.5}
v_t(x^*,T)=a(x^*,T)v_{xx}(x^*,T)+\left(c(x^*,T)-\lambda\right)v(x^*,T) <0,
\end{equation*}
which contradicts \eqref{prop3:34.1}. 
Hence, we conclude that $\left(x^*,t^*\right) \not\in D_T$, so that, from \eqref{prop3:30},
\begin{equation*}\label{prop3:35}
\left(x^*,t^*\right)=\left(0,t^*\right) \in \partial{D^T_L},
\end{equation*}
and, via \eqref{prop3:31},
\begin{equation}\label{prop3:36}
v_x(0,t^*) \leq 0.
\end{equation}
However, via \eqref{prop3:18}, \eqref{prop3:23} and \eqref{prop3:24.5}, we have,
\begin{equation*}\label{prop3:37}
v_x(0,t^*)=\frac{kv(0,t^*)}{D(u(0,t^*))}>0,
\end{equation*}
using (D2) and \eqref{prop3:31}, which contradicts \eqref{prop3:36}. 
Thus, we concluded that, 
\begin{equation*}\label{prop3:38}
\left(x^*,t^*\right) \not\in D_T \cup \partial{D^T_L},
\end{equation*}
which contradicts \eqref{prop3:30}. 
Hence, $v$ must be non-positive on $\overline{D}_{T}$, so that,
\begin{equation}\label{prop3:38.5}
v(x,t) \leq 0 \quad  \forall \  (x,t) \in \overline{D}_{T}. 
\end{equation}
It then follows, via \eqref{prop3:11} and \eqref{prop3:23}, that,
\begin{equation*}\label{prop3:39}
f_t(x,t) \leq 0 \quad \forall \ (x,t)\in \overline{D}_{T}.
\end{equation*} 
Therefore, from \eqref{prop3:4}, we have,
\begin{equation*}\label{prop3:40}
u_t(x,t) \leq 0 \quad \forall \ (x,t)\in \overline{D}_{T}.
\end{equation*}
Since this holds for any $T>0$ we have, 
\begin{equation*}
u_t(x,t) \leq 0 \quad \forall \ (x,t)\in \overline{D}_{\infty},
\end{equation*}
as required. 
\end{proof}
\noindent
As a consequence, we have the refinement,

\begin{corollary}\label{coro4}
Let  $u:\overline{D}_{\infty}\rightarrow\mathbb{R}$ be a solution to [IBVP]. 
Then,
\begin{equation*}
u_t(x,t) < 0 \quad \forall \ \left(x,t\right)\in {D}_{\infty}.
\end{equation*}
\end{corollary}

\begin{proof}
We introduce the function $\bar{v}:\overline{D}_{\infty}\rightarrow\mathbb{R}$ such that,
\begin{equation}\label{coro4:1}
v(x,t)=e^{{\mu}t}\bar{v}(x,t) \quad \forall \   \left(x,t\right) \in \overline{D}_{\infty},
\end{equation}
with $v:\overline{D}_{\infty}\rightarrow\mathbb{R}$ as introduced in \eqref{prop3:23} and with $\mu \in \mathbb{R}$ to be chosen. 
From \eqref{prop3:24} it follows that, 
\begin{equation}\label{coro4:2}
\bar{v} \in C^{2}\hspace{-0.5mm}\left(\overline{D}_{\infty}\right)\hspace{-0.5mm}.
\end{equation}
whilst from \eqref{prop3:38.5},
\begin{equation}\label{coro4:2.5}
\bar{v}(x,t) \leq 0 \quad \forall \   \left(x,t\right) \in \overline{D}_{\infty}.
\end{equation}
From \eqref{coro4:1} we have that for $\left(x,t\right) \in \overline{D}_{\infty}$,
\begin{align}
\label{coro4:3}
&v_x(x,t)=e^{\mu t}\bar{v}_x(x,t),
\\
\label{coro4:4}
&v_t(x,t)=\mu e^{\mu t}\bar{v}(x,t)+e^{\mu t}\bar{v}_t(x,t),
\\
\label{coro4:5}
&v_{xx}(x,t)=e^{\mu t}\bar{v}_{xx}(x,t).
\end{align}
Hence, after substituting \eqref{coro4:3}-\eqref{coro4:5} into \eqref{prop3:25}, we obtain,
\begin{equation}\label{coro4:6}
\bar{v}_t-a(x,t)\bar{v}_{xx}-\left(c(x,t)-\left(\lambda+\mu\right)\right)\bar{v}=0 \quad \forall \   \left(x,t\right) \in D_{\infty}.
\end{equation}
Setting $T>0$ and recalling \eqref{prop3:26} and \eqref{prop3:27}, we have,
\begin{equation*}\label{coro4:7}
c(x,t)-\left(\lambda+\mu\right) \geq -C_T - \left(2C_T+\mu\right)=-3C_T-\mu \quad  \forall \  (x,t) \in \overline{D}_T.
\end{equation*}
We now choose,
\begin{equation*}\label{coro4:8}
\mu=-3C_T,
\end{equation*}
so that,
\begin{equation}\label{coro4:9}
c(x,t)-\left(\lambda+\mu\right) \geq 0 \quad \forall \   \left(x,t\right) \in \overline{D}_{T}.
\end{equation}
Thus, from \eqref{coro4:6} and \eqref{coro4:9} together with \eqref{coro4:2.5}, we have,
\begin{equation}\label{coro4:10}
\bar{v}_t-a(x,t)\bar{v}_{xx} \leq 0 \text{ }\quad \forall \   \left(x,t\right) \in D_{T}.
\end{equation}
Now suppose there exists $\left(x^*,t^*\right)\in D_T$ such that, 
\begin{equation*}\label{coro4:11}
\bar{v}(x^*,t^*)=0= \sup_{(x,t)\in\overline{D}_T}\bar{v}(x,t).
\end{equation*}
Then, via \eqref{coro4:2}, we have that $\bar{v}$ is continuous on $\overline{D}_{T}$ and that $\bar{v}_x$, $\bar{v}_t$ and $\bar{v}_{xx}$ exist and are continuous on $D_T$. 
Hence, with \eqref{coro4:10}, we may apply the Strong Parabolic Maximum Principle \cite{WW} to $\bar{v}$ on $\overline{D}_T$, which requires, 
\begin{equation}\label{coro4:12}
\bar{v}(x,t)=0 \quad \forall \   \left(x,t\right) \in \overline{D}_{t*}.
\end{equation}
Then, using \eqref{coro4:12} together with \eqref{coro4:1}, \eqref{prop3:23}, \eqref{prop3:11}, \eqref{prop3:4} and (D2) we have, 
\begin{equation*}\label{coro4:13}
u_t(x,t)=0 \quad \forall \   \left(x,t\right) \in \overline{D}_{t*}.
\end{equation*}
It then follows, via \eqref{2NDBC2} and regularity conditions (R1) and (R2) that, 
\begin{equation*}\label{coro4:15}
u(x,t)=1 \quad \forall \   \left(x,t\right) \in \overline{D}_{t*}.
\end{equation*} 
However, this contradicts \eqref{2NDBC1}. 
Therefore,
\begin{equation*}\label{coro4:16}
\bar{v}(x,t)<0 \quad \forall \   \left(x,t\right) \in D_T,
\end{equation*} 
and, via \eqref{coro4:1}, \eqref{prop3:23}, \eqref{prop3:11}, \eqref{prop3:4} and (D2), we have,
\begin{equation*} \label{coro4:17}
u_t(x,t)<0 \quad \forall \   \left(x,t\right) \in D_T.
\end{equation*}
Since this holds for any $T>0$, we have,
\begin{equation*}\label{coro4:18}
u_t(x,t)<0 \quad \forall \   \left(x,t\right) \in D_{\infty},
\end{equation*}
as required. 
\end{proof}
\noindent
We next have,

\begin{proposition}\label{prop5}
Let $u:\overline{D}_{\infty}\rightarrow\mathbb{R}$ be a solution to [IBVP]. 
Then,
\begin{equation*}
1<D(u(x,t))u_x(x,t)<k \quad \forall \   \left(x,t\right) \in D_{\infty}.
\end{equation*}
\end{proposition}

\begin{proof}
Fix $t>0$ and let $\left(x,t\right)\in D_{\infty}$. 
Applying the mean value theorem, with $f:\overline{D}_{\infty}\rightarrow\mathbb{R}$ as defined in \eqref{prop3:2}, we have,
\begin{equation}\label{prop5:1}
f_x(x,t)-f_x(0,t)= f_{xx}(\theta x,t)x,
\end{equation}
and, 
\begin{equation}\label{prop5:2}
f_x(h(t),t)-f_x(x,t)=f_{xx}(x + \theta'(h(t)-x),t)(h(t)-x),
\end{equation}
with $0<\theta,\theta'<1$. 
Now, from \eqref{prop3:7}, we have,
\begin{equation*}\label{prop5:3}
\frac{\varepsilon}{D(u)}f_t=f_{xx} \quad \text{ on } D_{\infty}.
\end{equation*}
Thus, via (D2), \eqref{prop3:4} and Corollary \ref{coro4}, it follows that,
\begin{equation}\label{prop5:4}
f_{xx} < 0 \quad \text{ on } D_{\infty}.
\end{equation}
Consequently, via \eqref{prop5:1}, we have,
\begin{equation*}\label{prop5:5}
f_x(x,t)<f_x(0,t)\quad \forall \   \left(x,t\right) \in D_{\infty},
\end{equation*}
and, via \eqref{prop3:5},
\begin{equation}\label{prop5:6}
D(u(x,t))u_x(x,t)<D(u(0,t))u_x(0,t)\quad \forall \   \left(x,t\right) \in D_{\infty}.
\end{equation}
Next, using \eqref{2NDBC1} and Proposition \ref{prop1}, we have,
\begin{equation*}\label{prop5:7}
D(u(0,t))u_x(0,t)=ku(0,t) \leq k \quad \forall \   t \in \left(0,\infty\right).
\end{equation*}
Hence, via \eqref{prop5:6},
\begin{equation}\label{prop5:8}
D(u(x,t))u_x(x,t)<k \quad \forall \   \left(x,t\right) \in D_{\infty}.
\end{equation}
Similarly, via \eqref{prop5:2} and \eqref{prop5:4}, we have,
\begin{equation*}\label{prop5:9}
f_x(h(t),t)<f_x(x,t) \quad \forall \   \left(x,t\right) \in D_{\infty},
\end{equation*}
which, via \eqref{prop3:5}, becomes,
\begin{equation}\label{prop5:10}
D(u(x,t))u_x(x,t)>D(u(h(t),t))u_x(h(t),t)\quad \forall \   \left(x,t\right) \in D_{\infty}.
\end{equation}
Also, using \eqref{2NDBC3} and regularity condition (R1), we have,
\begin{equation*}\label{prop5:11}
D(u(h(t),t))u_x(h(t),t)=h_t(t)+1 \geq 1 \quad \forall \   t \in \left(0,\infty\right).
\end{equation*}
Hence, via \eqref{prop5:10},
\begin{equation*}\label{prop5:12}
D(u(x,t))u_x(x,t)>1 \quad \forall \   \left(x,t\right) \in D_{\infty},
\end{equation*}
and therefore, with \eqref{prop5:8}, we have, 
\begin{equation*}
1<D(u(x,t))u_x(x,t)<k \quad \forall \   \left(x,t\right) \in D_{\infty},
\end{equation*}
as required. 
\end{proof}
\noindent
The regularity condition (R2) with (D3) then immediately allows for,

\begin{corollary}\label{coro6}
Let  $u:\overline{D}_{\infty}\rightarrow\mathbb{R}$ be a solution to [IBVP]. 
Then,
\begin{equation*} \label{coro6:1}
1 \leq D(u(x,t))u_x(x,t) \leq k \quad \forall \   \left(x,t\right) \in \overline{D}_{\infty}.
\end{equation*}
\end{corollary}
\noindent
In addition, we have,

\begin{corollary}\label{coro7}
The existence of a solution to [IBVP] requires $k>1.$
\end{corollary}
\begin{proof}
Let $u: \overline{D}_{\infty}\rightarrow\mathbb{R}$ and $h:[0,\infty )\rightarrow\mathbb{R}$ be a solution to [IBVP]. 
Then from Proposition \ref{prop5} we have, 
\begin{equation*}
1<D(u(x,t))u_x(x,t)<k \quad \forall \   \left(x,t\right) \in D_{\infty},
\end{equation*}
which is only possible when $k>1$. 
\end{proof}
\noindent
Next we have,

\begin{proposition}\label{propnew}
Let $u:\overline{D}_{\infty}\rightarrow\mathbb{R}$ be a solution to [IBVP]. 
Then, 
\begin{align}
\nonumber & u(0,t)\geq \frac{1}{k} \quad \forall \  t \in (0,\infty), 
\\
\nonumber & u(x,t)>\frac{1}{k} \quad \forall \  (x,t) \in \overline{D}_{\infty} \setminus (\{0\}\times (0,\infty)).
\end{align}
\end{proposition}

\begin{proof}
From \eqref{2NDBC1} we have,
\begin{equation*}\label{new1}
ku(0,t)=D(u(0,t))u_x(0,t) \quad \forall \  t \in (0,\infty).
\end{equation*}
It then follows, from Corollary \ref{coro6}, that, 
\begin{equation}\label{new2}
u(0,t)\geq\frac{1}{k}\quad \forall \  t \in (0,\infty).
\end{equation}
In addition,
\begin{equation}\label{new3}
u(x,t)-u(0,t)=\int_{0}^{x} u_\lambda(\lambda,t)d\lambda>0 \quad \forall \  (x,t) \in \overline{D}_{\infty} \setminus (\{0\}\times [0,\infty)),
\end{equation}
via Corollary \ref{coro6} and (D2). 
Thus, from \eqref{new3} and \eqref{new2}, we have,
\begin{equation*}
u(x,t)>\frac{1}{k} \quad \forall \  (x,t) \in \overline{D}_{\infty} \setminus (\{0\}\times [0,\infty)).
\end{equation*} 
However, $u(0,0)=1> k^{-1}$, via Corollary \ref{coro7}, and so, 
\begin{equation*}\label{new4}
u(x,t)>\frac{1}{k} \quad \forall \  (x,t) \in \overline{D}_{\infty} \setminus (\{0\}\times (0,\infty)),
\end{equation*}
as required.  
\end{proof}
\noindent
We now obtain bounds on $h$ in the following,

\begin{proposition}\label{prop8}
Let $h:[0,\infty )\rightarrow \mathbb{R}$ describe the free boundary in [IBVP]. 
Then,
\begin{equation*}
0 < h(t) <F(1)-F\left( k^{-1} \right) \quad \forall \   t \in \left(0,\infty\right).
\end{equation*}
\end{proposition}

\begin{proof}
First, via \eqref{2NDBC5} and \eqref{uinC}, we have $h(0)=0$.
Also, from \eqref{2NDBC2} and \eqref{uinC} we have $u\left(0,0\right)=1$. 
Hence, via \eqref{uinC}, Corollary \ref{coro7}, \eqref{2NDBC1} and (D1) we have,
\begin{equation}\label{prop8:1}
D(1)u_x(0,0)=u_x(0,0)=ku\left(0,0\right)=k>1.
\end{equation}
Thus, via \eqref{2NDBC3}, \eqref{uinC} and \eqref{prop8:1},
\begin{equation}\label{prop8:2}
h_t(0)=u_x(0,0)-1=k-1>0.
\end{equation}
It then follows from \eqref{prop8:2}, with $h(0)=0$ and regularity condition (R1), that, 
\begin{equation}\label{prop8:3}
h(t)>0 \quad \forall \  t \in \left(0,\infty \right).
\end{equation}
Next take any $t>0$. 
It follows from the mean value theorem with \eqref{prop3:2}, \eqref{prop3:3} and \eqref{prop3:5} that,
\begin{align}
\nonumber 
F(u(h(t),t))-F(u(0,t)) & = F_x \left(u \left( \hat{\theta} h(t),t \right) \right) h(t) \\
& 
\label{prop8:4} 
= D\left(u\left(\hat{\theta}h(t),t\right)\right) u_x\left(\hat{\theta}h(t),t\right)h(t),
\end{align}
with $0<\hat{\theta}<1$. 
Then, via Proposition \ref{prop5}, \eqref{prop8:3} and \eqref{prop8:4},
\begin{equation*}\label{prop8:5}
F(u(h(t),t))-F(u(0,t))>h(t),
\end{equation*}
from which we obtain, via \eqref{2NDBC2},
\begin{equation}\label{prop8:6}
h(t)<F(1)-F(u(0,t)).
\end{equation}
However, via \eqref{2NDBC1} and Corollary \ref{coro6},
\begin{equation*}\label{prop8:6.5}
u(0,t) = \frac{D(u(0,t))u_x(0,t)}{k} \geq \frac{1}{k}.
\end{equation*}
Therefore, via \eqref{prop8:6} and \eqref{eq:s4definitionF},
\begin{equation}\label{prop8:7}
h(t) < F(1)-F(u(0,t)) \leq F(1)- F\left( k^{-1} \right) \quad \forall \  t \in \left(0,\infty \right).
\end{equation}
Thus, combining \eqref{prop8:3} and \eqref{prop8:7}, we have,
\begin{equation*}\label{prop8:8}
0 < h(t) < F(1)- F\left( k^{-1} \right) \quad \forall \  t \in \left(0,\infty \right),
\end{equation*}
as required.  
\end{proof}
\noindent
Now let $k>1$ and let $u:\overline{D}_{\infty}\rightarrow\mathbb{R}$ and $h:\left[0,\infty\right)\rightarrow\mathbb{R}$ be a solution to [IBVP]. 
It follows from \eqref{2NDBC3}, (R1) and Proposition \ref{prop8} that $h(t)$ is a monotone increasing function of $t$ and is bounded above by $F(1)-F\left( k^{-1}\right) >0$. 
Consequently, there exists a constant $0 < \bar{h} \leq F(1)-F\left( k^{-1}\right)$ such that,
\begin{equation*}
h(t) \rightarrow \bar{h} \text{ as } t \rightarrow \infty.
\end{equation*}
Similarly, Proposition \ref{prop1}, Proposition \ref{prop3}, Corollary \ref{coro6} and Proposition \ref{propnew} together with the Ascoli-Arzel\`a Compactness Theorem establish the existence of a continuous function $\bar{u}:[0,\bar{h}]\rightarrow \mathbb{R}$ such that, 
\begin{equation*}
u(x,t)\rightarrow \bar{u}(x) \text{ as } t \rightarrow \infty \text{ uniformly for } 0 \leq x \leq h(t),
\end{equation*}
with $\bar{u}(x)$ monotone increasing  for $x \in [0,\bar{h}]$, and,
\begin{equation*}
\frac{1}{k} \leq \bar{u}(x) \leq 1 \quad \forall \   x \in [0,\bar{h}].
\end{equation*}
Further, the bounds obtained on $u_x$, $u_t$ and $h_t$, and consequently bounds on $u_{xx}$ together with the Ascoli-Arzela Theorem, allow for a deduction that $\bar{u}_x$ and $\bar{u}_{xx}$ exist and are continuous, and moreover, $\bar{u}:[0,\bar{h}]\rightarrow \mathbb{R}$ and $\bar{h}$ must satisfy problem \eqref{SSPDE}-\eqref{SSBC3}, and so are steady state solutions to [IBVP]. 
Hence $\bar{u}=u_s$ and $\bar{h}=h_s$, as discussed in subsection \ref{steadystate}. 
It is convenient to summarize the results in this subsection in the following,

\begin{theorem}\label{theorem9}
The existence of a solution to [IBVP] requires $k > 1$. 
With $k > 1$, let $u:\overline{D}_{\infty}\rightarrow\mathbb{R}$ and $h:\left[0,\infty\right)\rightarrow\mathbb{R}$ be a solution to [IBVP].
Then,
\renewcommand{\labelenumi}{\roman{enumi}}
\begin{enumerate}[({\roman{enumi}})]
\item $\frac{1}{k}<u(x,t)<1 \quad \forall \   \left(x,t\right) \in D_{\infty}$,
\item $u_t(x,t) < 0 \quad \forall \ \left(x,t\right)\in {D}_{\infty}$,
\item $1<D(u(x,t))u_x(x,t)<k \quad \forall \   \left(x,t\right) \in D_{\infty}$,
\item $(D(u)u_x)_x<0 \quad \forall \ \left(x,t\right)\in {D}_{\infty},$
\item $0 < h(t) <F(1)-F\left( k^{-1} \right)$ and $0 \leq h_t(t) < k-1 \quad \forall \   t \in \left(0,\infty\right)$,
\item $h(t) \rightarrow h_s^{-} = F(1)-F \left( k^{-1} \right)$ as $t \rightarrow \infty$,
\item $u(x,t) \rightarrow u_s^{+}(x)= F^{-1}\left(x+F\left( k^{-1} \right)\right)$ as $t \rightarrow \infty$ uniformly for $0 \leq x \leq h(t)$. 
\end{enumerate}
\end{theorem}
\noindent
We recall that the limit in $(vi)$ is from below, whilst the limit in ($vii$) is from above. 
Also, we note that in physical terms the requirement that $k>1$, for a solution to [IBVP] to exist, requires that the cooling process has to be sufficiently strong in order for the development of a wax layer to initiate. 
From Theorem \ref{theorem9} ($i$) - ($v$) we have obtained a priori bounds for [IBVP] on $u:\overline{D}_\infty\to\mathbb{R}$, the partial derivative $u_x : {D}_\infty\to\mathbb{R}$, together with $h:(0,\infty )\to\mathbb{R}$. 
Consequently, both global existence and uniqueness for [IBVP], can be anticipated by adopting an iterative approach to accommodate the quasi-linear terms (see, for example \cite{LUS}), whilst following, in principle, Cannon and Hill \cite{Cannon1967}. 
We next develop the analysis of [IBVP] by considering the structure of the solution as $t \rightarrow 0^+$ and correspondingly as $t \rightarrow \infty$.

\section{Coordinate Expansions} \label{CoordExp}
\noindent
We begin this section by analysing the structure of the solution to [IBVP] as $t \rightarrow 0^+$. 
After which we consider the structure of the solution to [IBVP] as $t \rightarrow \infty$.

\subsection{Coordinate Expansions as {\normalfont{$t\rightarrow 0^+$}}}
\noindent
We consider the structure of the solution to [IBVP] (with $k>1$) as $t \rightarrow 0^+$. 
It follows, from \eqref{2NDBC1}-\eqref{2NDBC5}, that, 
\begin{equation*}\label{5.1-1}
h=O(t), \quad u=1+O(t), \quad x=O(t),
\end{equation*} 
as $t \rightarrow 0^+$. 
Therefore, we introduce the scaled coordinate,
\begin{equation}\label{5.1-2}
X=\frac{x}{t}=O(1) \quad \text{ as } t \rightarrow 0^+,
\end{equation}
and write,
\begin{align}
\label{5.1-3}
& h(t)=tH(t),
\\
\label{5.1-3.5}
& u(X,t)=1+tU(X,t),
\end{align}
with
\begin{align}
\label{5.1-4}
& H(t)=H_1+H_2t+O(t^2),
\\
\label{5.1-4.5}
& U(X,t)=U_1(X)+U_2(X)t+O(t^2),
\end{align}
as $t \rightarrow 0^+$ with $0 \leq X \leq H(t)$. 
In terms of $X$, $t$, $H$ and $U$, [IBVP] becomes,
\begin{align}
\label{5.1-5}
& \varepsilon t (U-XU_X+tU_t) = (D(1+tU)U_X)_X; && 0 <X<H(t), \quad t>0, 
\\
\label{5.1-6}
& D(1+tU)U_X = k\left(1+tU\right); && X =0, \quad t>0, 
\\
\label{5.1-7}
& U = 0; && X =H(t), \quad t>0, 
\\
\label{5.1-8}
& H+tH_t=D(1+tU)U_X-1; && X = H(t), \quad t>0,
\\
\label{5.1-9}
& U \text{ bounded as } t \rightarrow 0^+ \text{ uniformly for } && 0 \leq X \leq H(t),
\\
\label{5.1-10}
& H \text{ bounded as } t \rightarrow 0^+.
\end{align}
Substituting from \eqref{5.1-4} and \eqref{5.1-4.5} into \eqref{5.1-5}-\eqref{5.1-8} we have at leading order
\begin{align}
\label{5.1-15}
& U''_1=0; \hspace{20mm} 0<X<H_1,
\\
\label{5.1-16}
& U'_1(0)=k, 
\\
\label{5.1-17}
& U_1(H_1)=0, 
\\
\label{5.1-18}
& H_1=U'_1(H_1)-1.
\end{align}
The solution to this boundary value problem is readily obtained as,
\begin{equation}\label{5.1-25}
H_1=k-1>0,
\end{equation}
with
\begin{equation}\label{5.1-25.5}
U_1(X)= k(X-(k-1)) \quad \forall \ \hspace{1mm} 0 \leq X \leq k-1.
\end{equation}
Terms at $O(t)$ lead to the following boundary value problem for $H_2$ and $U_2$,
\begin{align}
\label{5.1-26}
& \left(D'(1)U_1U'_1+U'_2\right)' = \varepsilon \left(U_1-XU'_1 \right) ; \quad 0<X<H_1,
\\
\label{5.1-27}
& D'(1)U_1(0)U'_1(0)+U'_2(0)=kU_1(0),
\\
\label{5.1-28}
& U'_1(H_1)H_2+U_2(H_1)=0,
\\
\label{5.1-29}
& (2-U_1''(H_1))H_2 = D'(1)U_1(H_1)U'_1(H_1)+U'_2(H_1).
\end{align}
After some calculation, we obtain the solution to this boundary value problem as,  
\begin{align}
\label{DNE5.*}
H_2 & = -\frac{1}{2}k(k-1)(k+\varepsilon(k-1)) ,
\\
\nonumber 
 U_2(X) & = - \left( \frac{\varepsilon}{2} k(k-1) + \frac{D'(1)k^2}{2}\right) X^2 \\
\nonumber
& \quad \ \ \ - \left( k^2(k-1) - D'(1)k^2(k-1) \right) X \\
\nonumber
& \quad \ \ \ + k^2(k-1)^2 \left(\frac{\varepsilon - D'(1)}{2} + 1 \right) \\ 
\label{5.1-42}
 & \quad \ \ \ + \frac{k^3}{2}(k-1)+\frac{\varepsilon}{2}k(k-1)^3 \quad \forall \ \hspace{1mm} 0 \leq X \leq H_1.
\end{align}
Thus, the coordinate expansions \eqref{5.1-3} and \eqref{5.1-3.5} as $t \rightarrow 0^+$, via \eqref{5.1-4}, \eqref{5.1-4.5}, \eqref{5.1-25}, \eqref{5.1-25.5}, \eqref{DNE5.*} and \eqref{5.1-42}, are given by, 
\begin{equation}\label{5.1-44}
h(t) = (k-1)t - \frac{k}{2}(k-1)\left( \varepsilon (k-1) +k \right)t^2 +O(t^3) \quad \text{ as } t \rightarrow 0^+,
\end{equation} 
and,
\begin{align}
\nonumber 
u(X,t) & = 1+(kX-k(k-1))t + \bigg( - \left( \frac{\varepsilon}{2} k(k-1) + \frac{D'(1)k^2}{2}\right) X^2
\\ 
\nonumber
& \quad - \left( k^2(k-1) - D'(1)k^2(k-1) \right) X \\
\nonumber 
& \quad + k^2(k-1)^2 \left(\frac{\varepsilon - D'(1)}{2} + 1 \right) +\frac{k^3}{2}(k-1) +\frac{\varepsilon}{2}k(k-1)^3 \bigg) t^2
\\
\label{5.1-45}
& \quad + O(t^3) \text{ as } t \rightarrow 0^+ \text{ uniformly for } 0 \leq X \leq t^{-1}h(t).
\end{align}
We obtain, from \eqref{5.1-44} and \eqref{5.1-45}, that,
\begin{equation*}
h_t(t) \rightarrow (k-1) \quad \text{ as } t \rightarrow 0^+,
\end{equation*}
whilst recalling \eqref{5.1-2} we have,
\begin{equation} \label{JMNME1}
u_x(x,t) \rightarrow k \quad \text{ as } t \rightarrow 0^+ \text{ uniformly for } 0 \leq x \leq h(t),
\end{equation}
and,
\begin{equation} \label{JMNME2}
u_{xx}(x,t)\rightarrow -\varepsilon k(k-1)-D'(1)k^2 \quad\text{ as } t \rightarrow 0^+ \text{ uniformly for } 0 \leq x \leq h(t).
\end{equation}

\subsection{Coordinate Expansions as $t \to \infty$}
\noindent
We now consider the structure of the solution to [IBVP] as $t \rightarrow \infty$. 
From Theorem \ref{theorem9}, it follows that,
\begin{equation*}\label{5.2-1}
h=h_s+o(1), \quad u=u_s(x)+o(1),
\end{equation*} 
as $t \rightarrow \infty$. 
Thus we write,
\begin{align}
\label{5.2-2}
&h(t)=h_s+\bar{h}(t), 
\\
\label{5.2-2.5}
&u(x,t)=u_s(x)+\bar{u}(x,t),
\end{align}
with, 
\begin{align*}
&\bar{h}(t)=o(1)\text{ as }t \rightarrow \infty, 
\\
&\bar{u}(x,t)=o(1)\text{ as }t \rightarrow \infty \text{ uniformly for } 0 \leq x \leq h(t).
\end{align*}
On substituting from \eqref{5.2-2} and \eqref{5.2-2.5} into [IBVP] we obtain the leading order problem for $\bar{u}$ and $\bar{h}$, as
\begin{align}
\label{5.2-4}
& \varepsilon\bar{u}_t = \left(D(u_s(x))\bar{u}\right)_{xx}; && 0<x<h_s, \quad t\gg 1,
\\
\label{5.2-5}
& (D(u_s(x))\bar{u})_{x}=k\bar{u}; && x=0, \quad t \gg 1,
\\
\label{5.2-6}
& \bar{u}+\bar{h}=0; && x=h_s, \quad t \gg 1,
\\
\label{5.2-7}
& \bar{h}_t=\left(D(u_s(x))\bar{u}\right)_{x}; && x=h_s, \quad t \gg 1.
\end{align}
We can eliminate $\bar{h}(t)$ from \eqref{5.2-4}-\eqref{5.2-7} to obtain,
\begin{align}
\label{5.2-8}
& \varepsilon\bar{u}_t=\left(D(u_s(x))\bar{u}\right)_{xx}; && 0<x<h_s, \quad t\gg 1,
\\
\label{5.2-9}
& \left(D(u_s(x))\bar{u}\right)_{x}=k\bar{u}; && x=0, \quad t \gg 1,
\\
\label{5.2-10}
& \bar{u}_t=-\left(D(u_s(x))\bar{u}\right)_{x}; && x=h_s, \quad t \gg 1, 
\end{align}
after which $\bar{h}(t)$ is recovered as,
\begin{equation}\label{5.2-10.5}
\bar{h}(t)=-\bar{u}(h_s,t);\quad t\gg 1.
\end{equation}
We look for a solution to \eqref{5.2-8}-\eqref{5.2-10} in the form, 
\begin{align}
\label{5.2-11}
&\bar{u}(x,t)=\phi(x)e^{-\lambda t},
\end{align}
with $\lambda \in \mathbb{C}$ to be determined. 
After substituting from \eqref{5.2-11} into \eqref{5.2-8}-\eqref{5.2-10} we obtain the linear eigenvalue problem,
\begin{align}
\label{5.2-13}
& \left( D(u_s(x))\phi\right)_{xx} = -\varepsilon \lambda \phi; && 0 <x<h_s, 
\\
\label{5.2-14}
& \left( D(u_s(x))\phi\right)_{x}=k \phi; && x =0,
\\
\label{5.2-15}
& \lambda \phi = \left( D(u_s(x))\phi\right)_{x}; && x =h_s
\end{align}
with eigenvalue $\lambda \in \mathbb{C}$. 
It is now convenient to introduce the function $\psi : [0,h_s]\rightarrow \mathbb{R}$ defined by,
\begin{equation}\label{5.2-16}
\psi(x) = D(u_s(x))\phi(x) \quad \forall \  \hspace{1mm} 0 \leq x\leq h_s,
\end{equation}
together with the prescribed function $\Delta : [0, h_s]\rightarrow \mathbb{R}$ given by,
\begin{equation}\label{5.2-17}
\Delta(x)= D(u_s(x)) \quad \forall \  \hspace{1mm} 0 \leq x\leq h_s.
\end{equation}
Substituting from \eqref{5.2-16} into \eqref{5.2-13}-\eqref{5.2-15} we obtain the equivalent eigenvalue problem,
\begin{align}
\label{5.2-18}
& \psi'' + \frac{\varepsilon}{\Delta(x)} \lambda \psi=0; && 0<x<h_s,
\\
\label{5.2-19}
& \psi' = \frac{k}{D(k^{-1})} \psi; && x=0, 
\\
\label{5.2-20}
&\psi' = \lambda \psi; && x=h_s. 
\end{align}
Henceforth, we will refer to the generalised Sturm-Liouville eigenvalue problem given by \eqref{5.2-18}-\eqref{5.2-20} as [S-L]. 
We note that [S-L] has the following properties:
\begin{enumerate}
\item[(S1)] The eigenvalues of [S-L] are all real and may be written as,
\begin{equation*}\label{5.2-21}
{\lambda}_0 < {\lambda}_1 < {\lambda}_2 < \cdots <{\lambda}_r < \cdots
\end{equation*}
with ${\lambda}_r \rightarrow \infty$ as $r \rightarrow \infty$.
\item[(S2)] For $\lambda={\lambda}_r$ (where $r=0,1,2,\ldots$) then the corresponding eigenfunction is real valued, say, ${\psi}_r:[0,h_s] \rightarrow \mathbb{R},$ and may be normalised so that, $\psi_r(0)>0$ and,
\begin{equation}\label{5.2-21.5}
\int_{0}^{h_s}\frac{1}{\Delta(\sigma)}{\psi_r^2} (\sigma) d\sigma = 1.
\end{equation}
\item[(S3)] $\psi_r(x)>0$ for all $x \in [0,h_s]$ if and only if $r=0$.
\end{enumerate}
We are now able to establish the following,

\begin{proposition}
Let $\lambda_0 \in \mathbb{R}$ and ${\psi}_0:[0,h_s] \rightarrow \mathbb{R}$ be the zeroth eigenvalue and the corresponding zeroth normalised eigenfunction of [S-L]. 
Then $\lambda_0 >0$.
\end{proposition}

\begin{proof}
In [S-L] we will set $\lambda=\lambda_0$ and $\psi(x)=\psi_0(x)$, so that,
\begin{align}
\label{5.2-22}
& \psi''_0+\frac{\varepsilon}{\Delta(x)} \lambda_0 \psi_0 = 0; && 0<x<h_s, 
\\
\label{5.2-23}
& \psi'_0 = \frac{k}{D(k^{-1})}\psi_0; && x = 0, 
\\
\label{5.2-24}
& \psi'_0 = \lambda_0 \psi_0; && x =h_s.
\end{align}
We first multiply both sides of \eqref{5.2-22} by $\psi_0$ to obtain, after an integration,
\begin{equation}\label{5.2-27}
[\psi_0(x)\psi'_0(x)]_0^{h_s}-\int_{0}^{h_s}\left(\psi'_0(x)\right)^2dx + \varepsilon\lambda_0=0,
\end{equation}
where use has been made of \eqref{5.2-21.5}. 
With use of \eqref{5.2-23} and \eqref{5.2-24}, we obtain from \eqref{5.2-27}, 
\begin{equation*}\label{5.2-29}
\psi^2_0(h_s) \lambda_0- \frac{k}{D(k^{-1})}\psi^2_0(0)- \int_{0}^{h_s}\left(\psi'_0(x)\right)^2dx + \varepsilon\lambda_0 = 0,
\end{equation*}
and so, after rearranging, 
\begin{equation*}\label{5.2-30}
\lambda_0 = \frac{\left( k D(k^{-1})^{-1}\psi^2_0(0) + \int_{0}^{h_s}\left(\psi'_0(x)\right)^2dx \right)}{\left(\psi^2_0(h_s) + \varepsilon \right)}.
\end{equation*}
Hence, via (D2) and (S2) we have $\lambda_0>0$, as required. 
\end{proof}
\noindent
It is now instructive to consider [S-L] as $\varepsilon \rightarrow 0$. 
We write the eigenvalues of [S-L] as
\begin{equation*}\label{5.2:1}
0<\lambda_0(\varepsilon)<\lambda_1(\varepsilon)<\lambda_2(\varepsilon)<\cdots<\lambda_r(\varepsilon)<\cdots
\end{equation*}
where $\lambda_r(\varepsilon) \rightarrow \infty$ as $r \rightarrow \infty$. 
There are two possibilities as $\varepsilon \rightarrow 0$:
\begin{equation*}
\textrm{(i)}\ \lambda=O(1), \text{ or, } \textrm{(ii)}\ \lambda=O(\varepsilon^{-1}).
\end{equation*}
We first consider case (i). 
Hence we introduce the expansions,
\begin{align}
\label{5.2:2}
& \psi(x,\varepsilon) = \chi_0(x)+\varepsilon\chi_1(x)+O(\varepsilon^2),\quad x\in [0,h_s],
\\
\label{5.2:3}
& \lambda(\varepsilon)=l_0+\varepsilon l_1+O(\varepsilon^2),
\end{align}
as $\varepsilon \rightarrow 0$. 
After substituting from \eqref{5.2:2} and \eqref{5.2:3} into [S-L], we obtain at leading order,
\begin{align}
\label{5.2:7}
& \chi_0''=0; && 0<x<h_s, 
\\
\label{5.2:8}
& \chi_0' = \frac{k\chi_0}{D(k^{-1})}; && x =0, 
\\
\label{5.2:9}
& \chi_0' = l_0\chi_0; && x = h_s. 
\end{align}
On integrating \eqref{5.2:7} we obtain, 
\begin{equation}\label{5.2:10}
\chi_0(x)=\alpha x + \beta,
\end{equation}
where $\alpha, \beta \in \mathbb{R}$ are constants of integration. 
Applying condition \eqref{5.2:8} we have,
\begin{equation*}\label{5.2:11}
\alpha - \frac{k}{D(k^{-1})}\beta = 0,
\end{equation*}
and \eqref{5.2:10} becomes,
\begin{equation}\label{5.2:11.5}
\chi_0(x) = \alpha\left(x+\frac{D(k^{-1})}{k}\right) \quad \forall \  \hspace{1mm} 0 \leq x \leq h_s.
\end{equation}
Finally applying condition \eqref{5.2:9} requires,
\begin{equation*}
\left( l_0 \left( h_s+\frac{D(k^{-1})}{k} \right) -1 \right) \alpha = 0,
\end{equation*}
and so, for a non-trivial solution ($\alpha \neq 0$) we have,
\begin{equation}\label{5.2:14}
l_0 = \frac{k}{\left( D(k^{-1})+kh_s\right)}.
\end{equation} 
The constant $\alpha$ is now fixed via the normalization condition \eqref{5.2-21.5} as 
\begin{equation}\label{5.2:18}
\alpha=\left(\int_0^{h_s}\frac{1}{\Delta(s)}\left(s+\frac{D(k^{-1})}{k} \right)^2ds\right)^{-\frac{1}{2}}\hspace{-0.5mm}.
\end{equation}
We observe from \eqref{5.2:11.5} and (D2) that $\chi_0(x)>0$ for all $x\in [0,h_s]$. 
Thus in \eqref{5.2:2} and \eqref{5.2:3} it follows from (S3) that $\psi=\psi_0$ and $\lambda=\lambda_0$, and so we have constructed the lowest eigenvalue and eigenfunction only in case (i). 
It follows from (S1) that all higher eigenvalues will fall into case (ii), and need not be considered further. 
In summary, we have, from \eqref{5.2:3} and \eqref{5.2:14}, that,
\begin{equation}\label{5.2:36}
\lambda_0(\varepsilon) = \frac{k}{(D(k^{-1}) +kh_s)} + O(\varepsilon ),
\end{equation}
as $\varepsilon\rightarrow 0$, whilst from \eqref{5.2:2} and \eqref{5.2:11.5}, we have
\begin{equation}\label{5.2:37.5}
\psi_0(x,\varepsilon)=\alpha\left(x+\frac{D(k^{-1})}{k}\right)+O(\varepsilon),
\end{equation} 
as $\varepsilon \rightarrow 0$ uniformly for $x \in [0,h_s]$, with the positive constant $\alpha$ given by \eqref{5.2:18}. 
We observe that \eqref{5.2:36} gives,
\begin{equation}\label{5.2:38}
\lambda_0(\varepsilon)\sim \frac{k}{\left(D(k^{-1})+(k-1)\left<D\left(\left[k^{-1},1\right]\right)\right>\right)},
\end{equation}
as $\varepsilon \rightarrow 0$, where $\left<D\left(\left[k^{-1},1\right]\right)\right>$ is as defined in \eqref{S4:eq3.5}. 
Equation \eqref{5.2:38} highlights the contribution of the variable diffusivity across the steady state layer and in particular the mean of this diffusivity together with the diffusivity closest to the coolant. 
Finally, returning to \eqref{5.2-2} and \eqref{5.2-2.5} via \eqref{5.2-10.5}, \eqref{5.2-11} and \eqref{5.2-16} (with (S1)-(S3)) we have,
\begin{align}
\label{DNE4.xxxa}
u(x,t) & = u_s(x) + u_\infty \frac{\psi_0(x)}{D(u_s(x))} e^{-\lambda_0 t} + o(e^{-\lambda_0 t}) \quad \forall \  \hspace{1mm} 0 \leq x \leq h_s,
\\
\label{DNE4.xxxb}
 h(t) & = h_s - u_\infty \psi_0(h_s) e^{-\lambda_0 t} + o(e^{-\lambda_0 t}),
\end{align}
as $t \rightarrow \infty$, with $u_\infty$ being a positive (via Theorem \ref{theorem9} (vi) and (vii)) global constant, which remains undetermined in this large-$t$ analysis. 
The structure of $\lambda_0$ and $\psi_0(x)$ as $\varepsilon \rightarrow 0$ is given by \eqref{5.2:36}and \eqref{5.2:37.5}. 
The steady state solution is approached through terms exponentially small in $t$, with exponent $\lambda_0$, as $t\to \infty$.

\section{Asymptotic Solution to [IBVP] as $\mathbf{\varepsilon} \mathbf{\rightarrow} \mathbf{0}$} \label{kasymptotics}
\noindent
In many applications the parameter $\varepsilon$ is small (see, for example \cite{KL1} and \cite{Needham2}). 
Therefore, it is of value to consider [IBVP] as a parameter perturbation problem with $0<\varepsilon \ll 1$, and consider it's asymptotic solution as $\varepsilon \rightarrow 0$, with $k>1$. 
We expand the solution to [IBVP] in the form,
\begin{align}
\label{eq6:1}
u(x,t) & = u_0(x,t)+\varepsilon u_1(x,t)+O(\varepsilon^2), \\
\label{eq6:2}
h(t) & = h_0(t)+\varepsilon h_1(t)+O(\varepsilon^2),
\end{align}
as $\varepsilon\rightarrow 0$ with $x,t=O(1)$. 
On substituting from \eqref{eq6:1} and \eqref{eq6:2} into [IBVP] we obtain the following problem at leading order for $u_0(x,t)$ and $h_0(t)$, namely,
\begin{align}
\label{eq6:3}
& (F(u_0))_{xx} =0; && 0<x<h_0(t), \quad t>0,
\\
\label{eq6:4}
& (F(u_0))_x=ku_0; && x=0, \quad t>0,
\\
\label{eq6:5}
& u_0=1; && x=h_0(t), \quad t>0,
\\
\label{eq6:6}
& 1+ (h_{0})_t=(F(u_0))_x; && x=h_0(t), \quad t>0,
\\
\label{eq6:7}
& u_0 \rightarrow 1 \text{ as } t \rightarrow 0^+ \text{ uniformly for }  &&  0 \leq x \leq h_0(t),
\\
\label{eq6:8}
& h_0 \rightarrow 0 \text{ as } t\rightarrow 0^+.
\end{align} 
where $F:\mathbb{R}\to\mathbb{R}$ is given by \eqref{eq:s4definitionF}.
An integration of \eqref{eq6:3} gives,
\begin{equation}\label{eq6:15}
F(u_0(x,t))=A(t)x+B(t);\quad 0 \leq x \leq h_0(t),\quad t>0, 
\end{equation}
where $A(t)$ and $B(t)$ are smooth functions of $t$ to be determined. 
Applying condition \eqref{eq6:4} we require,
\begin{equation}\label{eq6:17}
A(t) = kF^{-1}(B(t)) ; \quad t>0.
\end{equation}
Therefore, after rearranging \eqref{eq6:17}, we can rewrite \eqref{eq6:15} as,
\begin{equation}\label{eq6:18}
F(u_0(x,t)) = A(t)x+F\left( k^{-1} A(t)\right);\quad 0\leq x \leq h_0(t),\quad t>0.
\end{equation}
We next apply condition \eqref{eq6:5} to \eqref{eq6:18} to obtain 
\begin{equation}
\label{DNE5:18}
h_0(t) = G(A(t)); \quad t>0.
\end{equation}
where $G:\mathbb{R}^+\to\mathbb{R}$ is such that 
\begin{equation}
\label{DNE5:19}
G(\lambda ) = \frac{F(1)-F(k^{-1}\lambda )}{\lambda}\quad \forall \ \hspace{1mm} \lambda\in\mathbb{R}^+ .
\end{equation}
It is readily established that the following properties are satisfied:
\begin{enumerate}
\item[(G1)] $G\in C^3(\mathbb{R}^+)$.
\item[(G2)] $G(k)=0$.
\item[(G3)] $G(\lambda )>0$ for all $\lambda\in (0,k)$.
\item[(G4)] $G'(\lambda )<0$ for all $\lambda \in (0,k]$.
\end{enumerate}
Finally, we must apply condition \eqref{eq6:6}, which gives,
\begin{equation}
\label{DNE5:20}
h_{0\hspace{1pt}t}(t) = A(t) - 1, \quad t>0.
\end{equation}
We observe, from Theorem \ref{theorem9} (v) and \eqref{DNE5:20}, that (recalling $k>1$)
\begin{equation}
\label{DNE5:21}
1 \leq A(t) \leq k \quad \forall \  t>0 . 
\end{equation} 
Thus, via (G4), \eqref{DNE5:21} and \eqref{DNE5:18}, we may deduce that,
\begin{equation}
\label{DEN5:22}
0 \leq h_0(t) \leq h_s(k) \quad \forall \  t>0,
\end{equation}
(which is in agreement with Proposition \ref{prop8}, recalling that $h_s(k) = F(1) - F(k^{-1})$) after which, via (G4), we may invert \eqref{DNE5:18} to obtain, 
\begin{equation}
\label{DEN5:23}
A(t) = G^{-1}(h_0(t)) , \quad t>0.
\end{equation}
Here $G^{-1}:[0,h_s(k)]\to [1,k]$ is such that $G^{-1}\in C^3([0,h_s(k)])$, and,
\begin{equation}
\label{DNE5:24}
G^{-1}(0)=k , \quad G^{-1} (h_s(k))=1  
\end{equation}
with
\begin{equation}
\label{DNE5:25}
G_\lambda^{-1}(\lambda ) < 0 \quad \forall \  \lambda \in [0,h_s(k)].
\end{equation}
Therefore, via \eqref{DNE5:20} and \eqref{DEN5:23} (together with the condition \eqref{eq6:8}) we require that $h_0:[0,\infty )\to\mathbb{R}$ satisfies the autonomous 1-dimensional dynamical system,
\begin{align}
\label{DNE5:26}
& h_{0\hspace{1pt}t} = G^{-1} (h_0)-1,\quad t>0, 
\\
\label{DNE5:27}
& h_0(t)\to 0^+ \text{ as } t\to 0^+.
\end{align}
This problem has a unique solution, say,
\begin{equation*}
h_0 = H(t),\quad t\geq 0, 
\end{equation*}
with $H\in C^3( [0,\infty ))$, and,
\begin{enumerate}
\item[(H1)] $H(0)=0$,
\item[(H2)] $H'(t) >0$ for all $t\in [0,\infty )$,
\item[(H3)] $H(t) = (k-1)t + O(t^2)$ as $t\to 0^+$,
\item[(H4)] $H(t)\sim h_s(k) - A_\infty (k) e^{-{\mu (k)t}}$ as $t\to\infty$.
\end{enumerate}
Here $A_\infty (k) >0$ is a global constant, and
\begin{equation}
\label{DNE5:28}
\mu (k) = \frac{k}{kh_s(k)+F'(k^{-1})} .
\end{equation}
We observe that (H4) and \eqref{DNE5:28} are in agreement with \eqref{DNE4.xxxb}. 
An implicit form for $H(t)$ is obtained from \eqref{DNE5:26} and \eqref{DNE5:27} as
\begin{equation}
\label{DEN5:29}
\int_{0}^{H(t)} \frac{d\lambda}{(G^{-1} (\lambda)-\lambda )} = t, \quad \forall \  t \geq 0.
\end{equation}
Finally, having determined $h_0(t)$, we obtain from \eqref{eq6:18},
\begin{equation}
\label{DNE5:30}
u_0(x,t) = F^{-1}(A(t)x + F(k^{-1}A(t)));\quad 0 \leq x \leq h_0(t),\quad t\geq 0 ,
\end{equation}
with $A(t)$ given in terms of $h_0(t)$ in \eqref{DEN5:23}. 
It is worth observing from \eqref{DNE5:30} that, 
\begin{equation*}
u_0(0,t) = k^{-1}A(t) = k^{-1} G^{-1}(h_0(t)), \quad t\geq 0, 
\end{equation*}
which, via \eqref{DNE5:24} and \eqref{DNE5:25}, is monotonic decreasing from unity to $k^{-1}$ with increasing $t$.

\section{Numerical Solution to [IBVP]} \label{numerics}

\noindent In this section we consider numerical solutions to [IBVP] for comparison with the theory of the previous sections.
For [IBVP] with constant diffusivity, the method of fundamental solutions was employed in \cite{Needham} to provide numerical approximation of the solution to [IBVP]{\footnote{The MATLAB files used to perform numerical simulations can be found \href{https://github.com/JCMUoB/The-Development-of-a-Wax-Layer-on-the-Interior-Wall-of-a-Circular-Pipe/tree/JCMUoB-Pre-print-\%2B-submitted-manuscript-version}{here}.}}. 
A useful feature of this method is that a node can be placed at $(x,t)=(0,0)$ in the domain to encapsulate the initial-boundary conditions \eqref{2NDBC4} and \eqref{2NDBC5}. 
However, for non-constant diffusivity, the partial differential equation \eqref{2NDPDE} is quasi-linear, and consequently, the nonlinearity precludes numerical methods based on fundamental solutions. 
Hence, in the current situation, we employ an explicit finite-difference method to provide numerical approximations to [IBVP]. 
We note that although this method is simple to apply, in this setting there are several limitations, primarily due to the representation of conditions \eqref{2NDBC4} and \eqref{2NDBC5}.
We first transform [IBVP] to a rectangular domain, by introducing
\begin{equation}
\label{JMNM1}
u(x,t)=v(X,t) , \quad X=\frac{x}{h(t)} \quad \forall \ (x,t)\in D_\infty \cup \partial D_\infty. 
\end{equation}
It then follows from \eqref{2NDPDE}-\eqref{2NDBC3} that [IBVP] becomes
\begin{align}
\label{JMNM2}
& \varepsilon h^2 v_t = D(v)v_{XX} + \varepsilon X h_t h v_X + D'(v)v_X^2; && 0 <X<1, \quad t>0, 
\\
\label{JMNM3}
& D(v)v_X=kvh; && X=0, \quad t >0,
\\   
\label{JMNM4}
& v=1; && X =1, \quad t >0,
\\
\label{JMNM5}
& hh_t = v_X - h \geq 0; && X = 1, \quad t >0.
\end{align}
whilst conditions \eqref{2NDBC4} and \eqref{2NDBC5} are extended, via \eqref{5.1-44} and \eqref{JMNME1}-\eqref{JMNME2} to
\begin{align}
\label{JMNM6}
h(t) & \sim (k-1)t - \frac{k}{2}(k-1)\left( \varepsilon (k-1) +k \right)t^2 \text{ as }t\to 0^+ 
\\
\label{JMNM7}
v(X,t) & \sim 1 + kh(t)(X-1) - \frac{k}{2}h^2(t)(\varepsilon (k-1) - D'(1)k)(X-1)^2 \text{ as }t\to 0^+ ,
\end{align} 
for $X\in [0,1]$. 
Due to the degeneracy of \eqref{JMNM2} at $t=0$, we set the initial conditions for the numerical method at $t=\delta>0$, with $\delta$ sufficiently small so that we can use the asymptotic forms for $h$ and $v$ in \eqref{JMNM6} and \eqref{JMNM7} at $t=\delta$, respectively.
We refer to the initial-boundary value problem given by \eqref{JMNM2}-\eqref{JMNM7} as [IBVP$^*$].

To implement the finite-difference method, we employed a uniform spatial grid with $N_x$ grid points to represent the interval $[0,1]$ so that the $i^{\text{th}}$ spatial grid-points $X_i = (i-1)dX$ with $dX=1/(N_x-1)$. 
The temporal grid points $t_j$, used to represent $[\delta , T]$, were not uniformly spaced, with the time step $dt$ chosen sufficiently small at each step to accommodate numerical stability conditions on the discrete evolution equations for $v_{i,j}\approx v(X_i,t_j)$ and $h_{j}\approx h(t_j)$, namely
\begin{equation*}
dt \ll \min\left\{ \frac{(hdX)^2\varepsilon}{2\sup_{X\in[0,1]}{D(v(X,t))}} ,\ \frac{2(hdX)^2 \varepsilon}{\sup_{X\in[0,1]}{D'(v(X,t))}} ,\ \frac{hdX}{(1-k)} \right\} .
\end{equation*}
We note that this local stability condition limiting $dt$ relaxes, as $h$ increases.
The discretization of \eqref{JMNM5}, \eqref{JMNM4}, \eqref{JMNM2} and \eqref{JMNM3} respectively, in the order that $h_{j+1}$ and $v_{i,j+1}$ are computed, is given by:
\begin{align*}
& v_{N_x,j} = 1 ; \\ 
& h_{j+1} = h_j + dt\left( \frac{v_{N_x,j} - v_{N_{x}-1,j}}{dx h_j}-1\right) ; \\ 
& v_{i,j+1} = v_{i,j} + \frac{dt}{\varepsilon h_j^2} \Bigg( D(v_{i,j})\left( \frac{ v_{i-1,j}-2v_{i,j}+ v_{i+1,j}}{dx^2} \right) \\
& \hspace{2cm} + \varepsilon X_i h_j \left( \frac{h_{j+1}-h_j}{dt}\right) \left( \frac{v_{i+1,j}-v_{i-1,j}}{2dx} \right) \\
& \hspace{2cm} + D'(v_{i,j}) \left( \frac{v_{i+1,j}-v_{i-1,j}}{2dx} \right)^2 \Bigg) \quad i=2 , \ldots , N_x -1; \\
& v_{1,j+1} = v_{2,j+1} - \frac{kdx h_{j+1} v_{2,j+1}}{D(v_{2,j+1})}.
\end{align*} 
\begin{figure}[h!]
\begin{subfigure}{.33\textwidth}
  \centering
  \includegraphics[width=\linewidth]{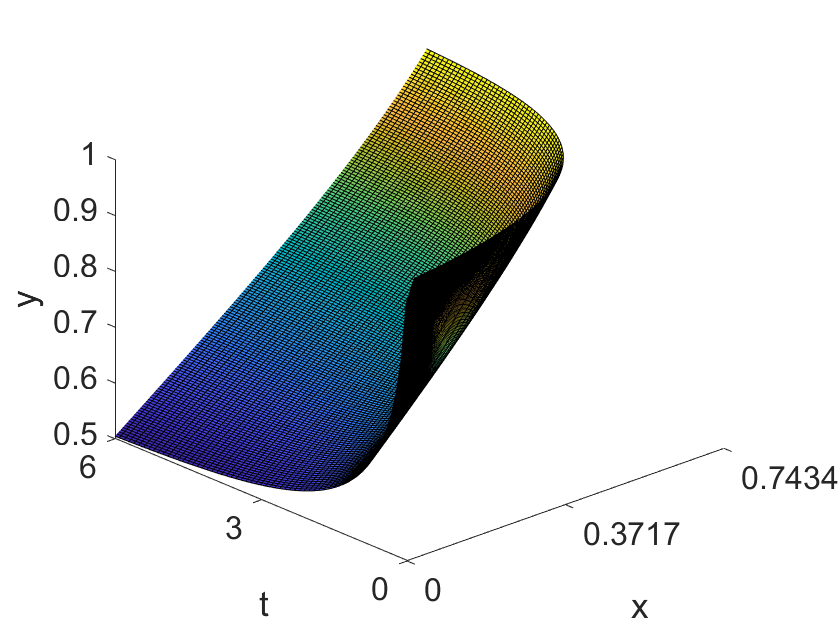}
  \caption{}
  \label{fig:1a}
\end{subfigure}%
\begin{subfigure}{.33\textwidth}
  \centering
  \includegraphics[width=\linewidth]{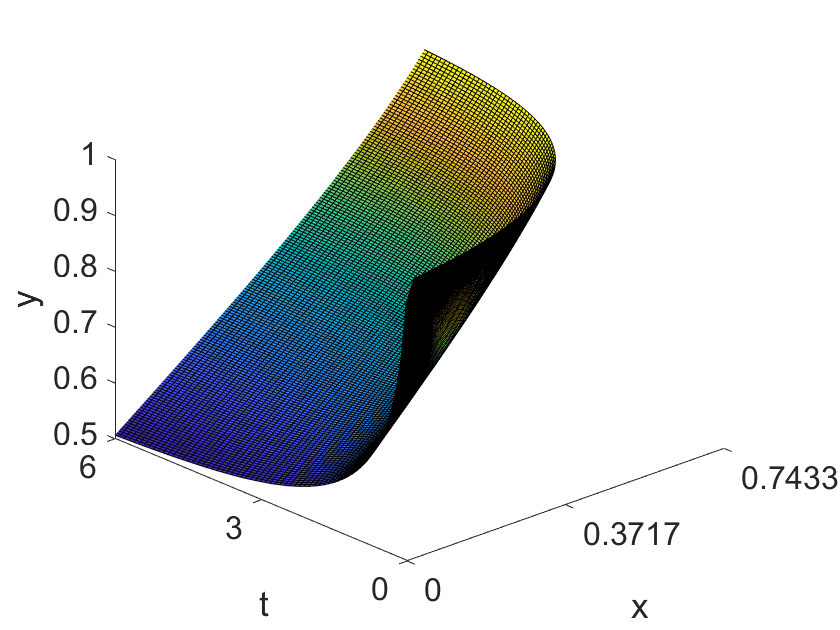}
  \caption{}
  \label{fig:1b}
\end{subfigure}
\begin{subfigure}{.33\textwidth}
  \centering
  \includegraphics[width=\linewidth]{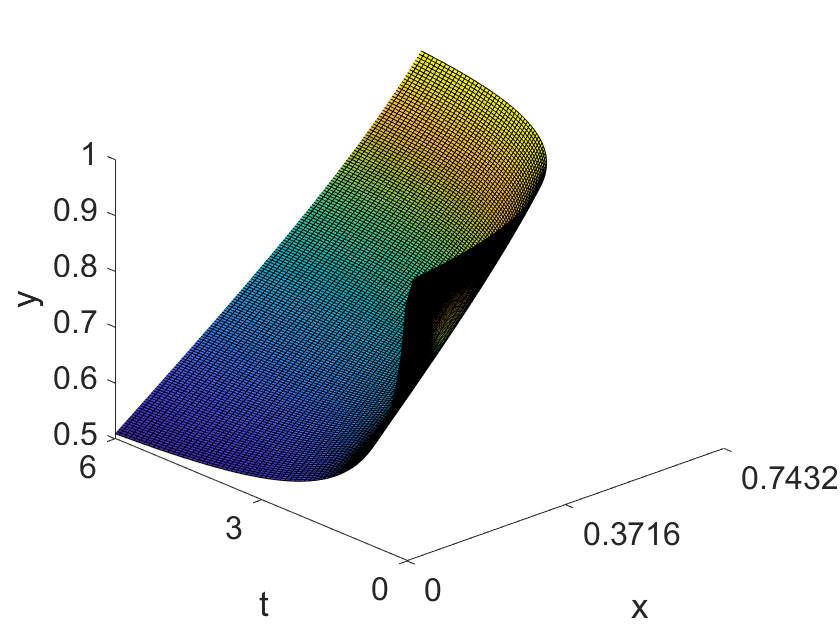}
  \caption{}
  \label{fig:1c}
\end{subfigure}
\\
\begin{subfigure}{.33\textwidth}
  \centering
  \includegraphics[width=\linewidth]{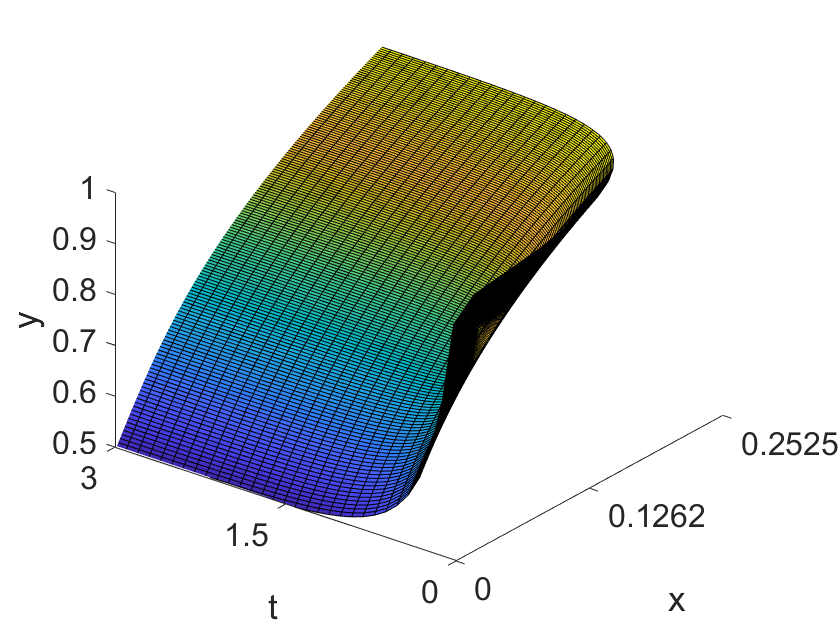}
  \caption{}
  \label{fig:1d}
\end{subfigure}%
\begin{subfigure}{.33\textwidth}
  \centering
  \includegraphics[width=\linewidth]{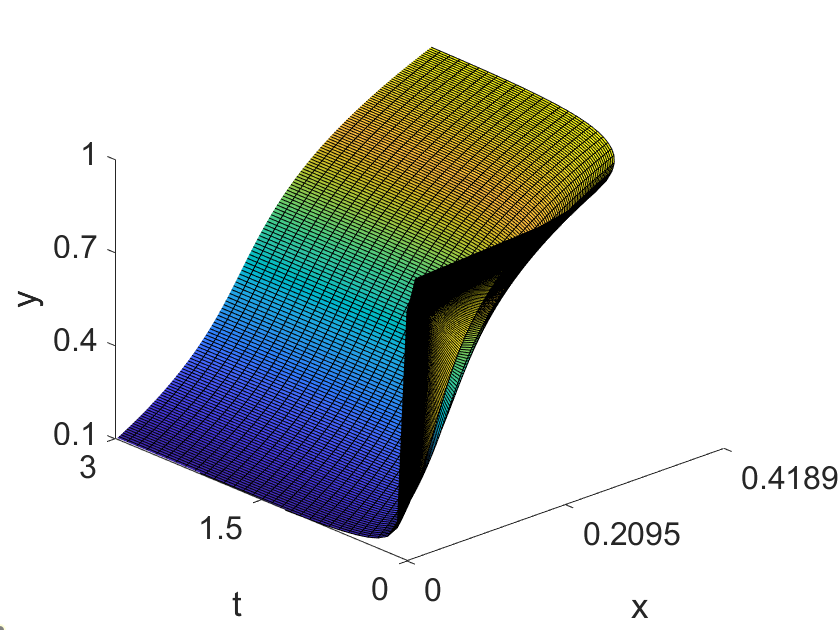}
  \caption{}
  \label{fig:1e}
\end{subfigure}
\begin{subfigure}{.33\textwidth}
  \centering
  \includegraphics[width=\linewidth]{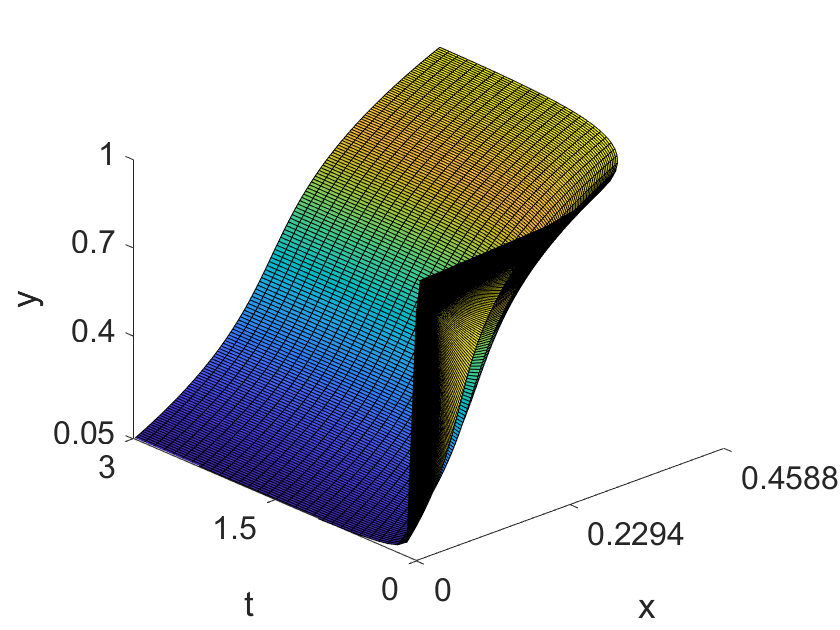}
  \caption{}
  \label{fig:1f}
\end{subfigure}
\caption{
Plots of $y=u(x,t)$ for $0 \leq x \leq h(t)$ with $u$ and $h$ being the numerical solution to [IBVP], obtained from $v_{i,j}$ and $h_j$.
In (a)-(c), we illustrate surface plots for: $D(u)=1+3u(1-u)$ and $k=2$, with $\varepsilon = 0.1$, $0.5$ and $1$ from left to right respectively.
In (d)-(f), we illustrate surface plots for: $D(u)=1-3u(1-u)$ and $\varepsilon = 0.5$, with $k = 2$, $10$ and $20$ from left to right respectively.
}
\label{fig:1}
\end{figure}
Using this numerical scheme we approximated the solution to [IBVP$^*$] in the cases when $D(v) = 1 + cv(1-v)$, with $c=3$, $2$, $1$, $0$, $-1$, $-2$ and $-3$.
The parameter $k(>1)$ and $\varepsilon(>0)$ were chosen from $k=2$, $10$ and $20$ with $\varepsilon = 0.1$, $0.5$ and $1$.
For the discretization we found that $N_x=161$ was sufficient to achieve good accuracy except for six cases, where to maintain a suitable level of accuracy, we used $N_x=201$.  
Numerical approximations to [IBVP$^*$] were obtained on domains which extended up to $t=20$, and in all cases, monotone convergence to the steady state solution was observed.  
Moreover, in all cases, the discrepancy between the numerical approximations for $v$ and $h$ evaluated at the final time level, with their respective steady state solutions, was less than $0.01$.
The 63 simulations took approximately 8 hours to run on a standard laptop with an Intel(R) Core(TM) i5-7Y54 CPU @ 1.20GHz 1.60 GHz processor.
See Summary Data{\footnote{Summary Data from numerical simulations is available \href{https://github.com/JCMUoB/The-Development-of-a-Wax-Layer-on-the-Interior-Wall-of-a-Circular-Pipe/tree/JCMUoB-Pre-print-\%2B-submitted-manuscript-version}{here}.}} for details related to run-time and `numerical convergence' times. 
\begin{figure}
\begin{subfigure}{.5\textwidth}
  \centering
  \includegraphics[width=\linewidth]{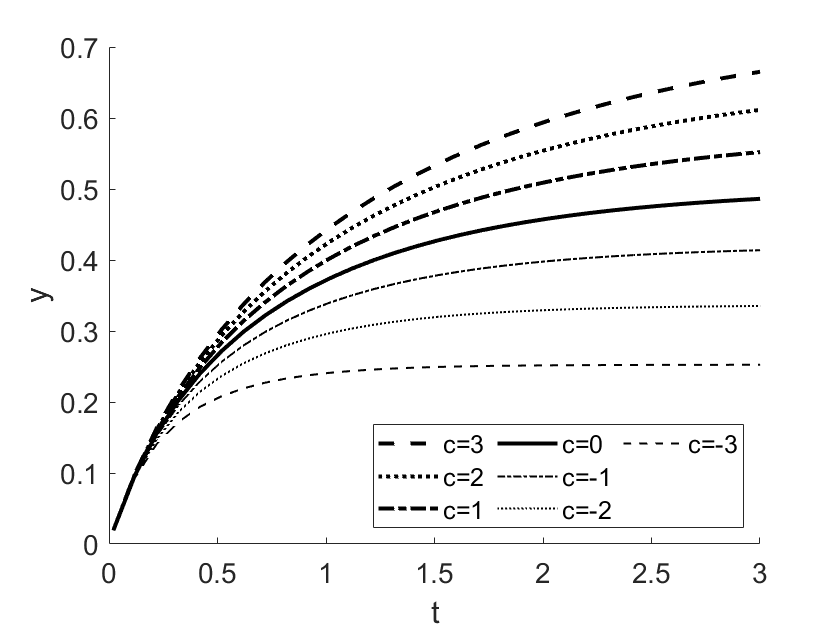}
  \caption{}
  \label{fig:2a}
\end{subfigure}
\begin{subfigure}{.5\textwidth}
  \centering
  \includegraphics[width=\linewidth]{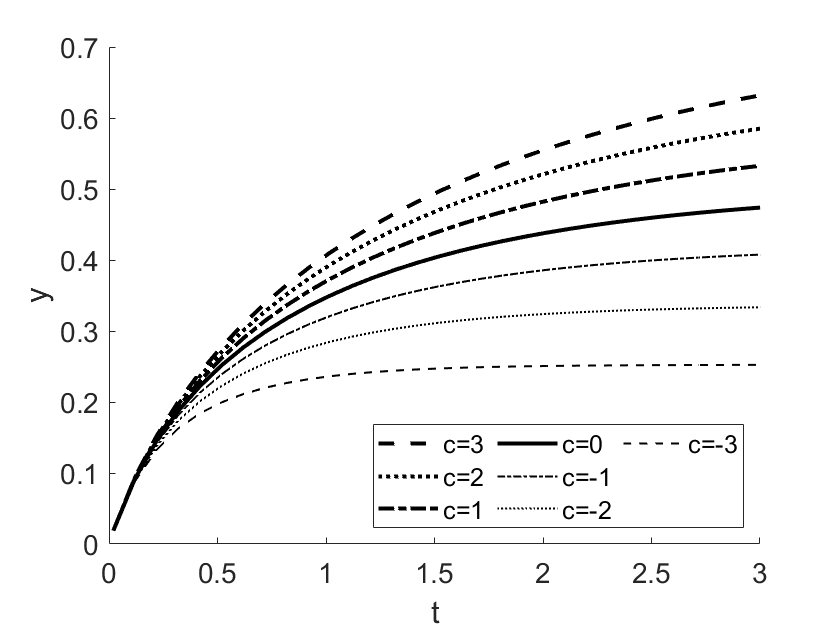}
  \caption{}
  \label{fig:2b}
\end{subfigure}
\\
\begin{subfigure}{.5\textwidth}
  \centering
  \includegraphics[width=\linewidth]{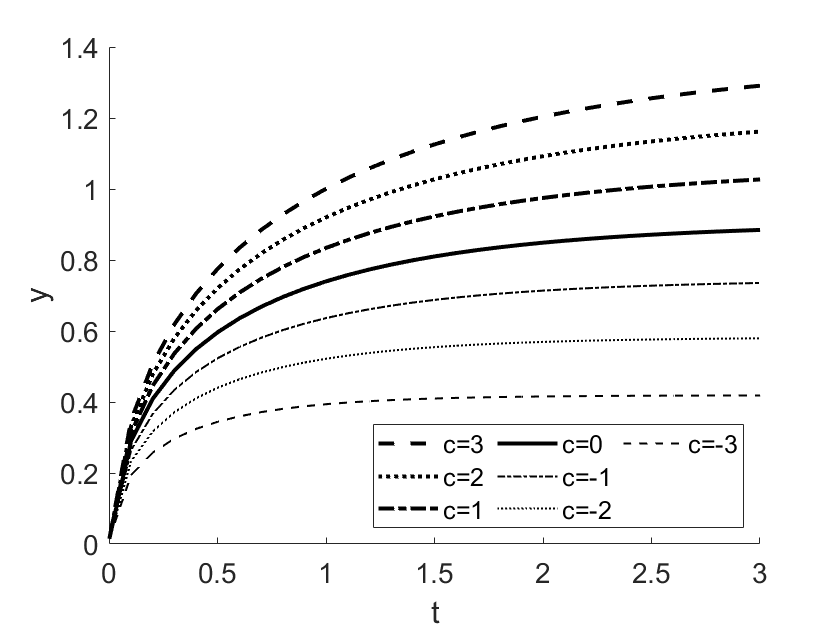}
  \caption{}
  \label{fig:2c}
\end{subfigure}
\begin{subfigure}{.5\textwidth}
  \centering
  \includegraphics[width=\linewidth]{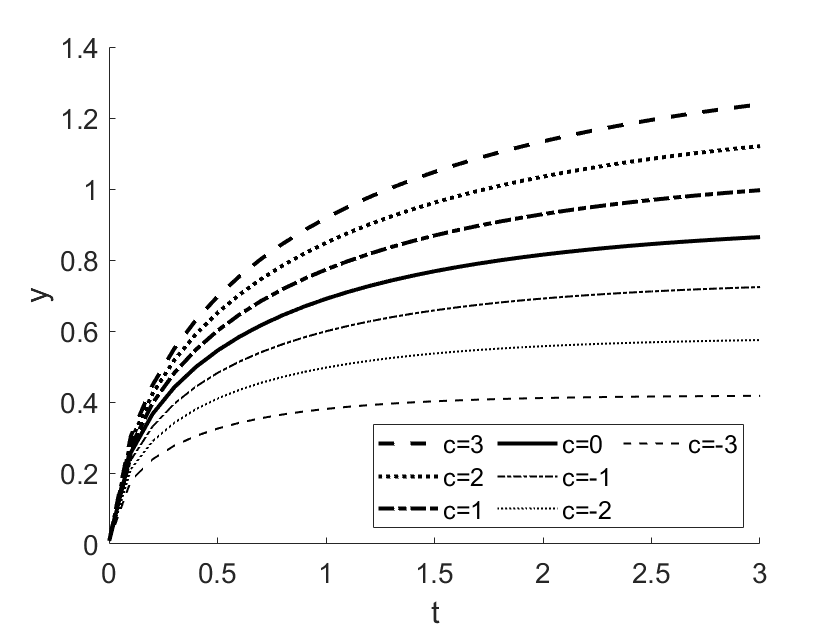}
  \caption{}
  \label{fig:2d}
\end{subfigure}
\caption{
Plots of the free boundary component of the numerical solution to [IBVP$^*$], $y=h(t)$.
Time domains are chosen to highlight the discrepancy in time taken for $h$ to get close to its steady state $h_s$.
In each of (a)-(d) above, we consider [IBVP] with $D(u)=1+cu(1-u)$ with 
(a) $\varepsilon = 0.1$ and $k=2$;
(b) $\varepsilon = 1$ and $k=2$;
(c) $\varepsilon = 0.1$ and $k=10$; and
(d) $\varepsilon = 1$ and $k=10$.
}
\label{fig:2}
\end{figure}
We observe in Figures \ref{fig:1}(a-f) various qualitative features of solutions to [IBVP].
Specifically, in Figures \ref{fig:1}(a-c), we observe that as $\varepsilon$ increases, with $D$ and $k$ fixed, the solutions to [IBVP], converge to the steady state $u_s$ at a slower exponential rate. 
This is further illustrated in Figures \ref{fig:2}(a,b) and Figures \ref{fig:2}(c,d), which graph $h$ against $t$ for each of the cases illustrated in Figures \ref{fig:1}(a-f).
Additionally, in Figures \ref{fig:1}(d-f) we observe that as $k$ increases, with $D$ and $\varepsilon$ fixed, firstly, that the steady state curvature changes from monotone to non-monotone as $k$ increases past $k=2$ and secondly, that as $k$ increases, the time taken for $u$ to approach the steady state $u_s$ decreases.

\section{Comparison with Experiments} \label{thevsexp}

Detailed experiments on the formation of wax deposited layers in straight circular pipes transporting heated oil, when subject to external wall cooling, have been performed and reported by Hoffman and Amundsen \cite{HA1}. 
In this section we make qualitative and trend comparisons with the theory presented in this paper and the experimental results in \cite{HA1}. 
We observe that the experiments reported in \cite{HA1} have fixed paraffinic wax and oil properties in all experiments, while the bulk oil temperature $T_o$, the coolant temperature, $T_c$ and the oil flow Nusselt number $N_u$ are varied in turn, with a series of experiments performed in each case and wax layer evolving profile and equilibrium thickness measured. 
First, we relate the variations in $T_o$, $T_c$ and $N_u$, respectively, to the dimensionless parameters in the model, namely $\varepsilon$ and $k$.
We observe from \eqref{epsilonandk} that, first $k>1$ for wax layer formation to occur, and thereafter,
\begin{enumerate}
\item[(a)] Increasing $T_o$ decreases $k$ whilst keeping $\varepsilon$ fixed.
\item[(b)] Decreasing $T_c$ increases $k$ and $\varepsilon$ together.
\item[(c)] Increasing $N_u$ decreases $k$ whilst leaving $\varepsilon$ fixed.
\end{enumerate}
Thus, via Section \ref{numerics}, we can refer, for comparison, to the evolution of the dimensionless wax layer thickness with dimensionless time in Figures \ref{fig:2}(a-d). 
Figure \ref{fig:2}(a,c) correspond to cases (a) and (c) above. 
Conversely Figures \ref{fig:2}(a,d) correspond to case (b) above. 

We first consider Figure 11 in \cite{HA1}. 
This graphs the experimental wax layer equilibrium height against the wall temperature $T_c$, at two different values of $N_u$. 
Each graph has a critical value of $T_c$, above which a wax layer does not form. 
This is consistent with the theory, corresponding to the critical value $k=1$. 
As $T_c$ decreases from the critical value, the equilibrium wax layer thickness increases; this is born out for the theory Figures \ref{fig:2}(a,c), where Figure \ref{fig:2}(c) has larger $k$ and $\varepsilon$ values than Figure \ref{fig:2}(a). 

We now consider Figure 15 in \cite{HA1}, which graphs the evolution of the wax layer thickness with time, for a number of increasing oil temperatures $T_o$. 
These profiles show remarkably similar structure to those theoretical profiles in Figures \ref{fig:2}(a-d). 
The experimental graphs show a lowering of the profile with increasing $T_o$. 
These can be compared with Figures \ref{fig:2}(a,c), which show lowering profiles with decreasing $k$. 
This is consistent via (a). 
Note that the experiments show that a wax layer will not form for $T_o$ sufficiently large. 
This corresponds to $k$ decreasing to $k=1$ in the theory. 

Finally we consider Figure 18 in \cite{HA1}. 
This shows the evolution of the wax layer thickness with time, for a number of Nusselt numbers, with $N_u$ increasing. 
We see that the wax layer profiles are lowering with increasing Nusselt number. 
This corresponds to case (c) in the theory, and we compare with Figures \ref{fig:2}(b,d). 
We see that, with $\varepsilon$ fixed, the corresponding profiles are lowering as $k$ decreases, in line with the experiments. 
To conclude, we note the striking similarity between the experimental profiles in Figure 18 of \cite{HA1} and the model profiles given in Figure \ref{fig:2}.

\section{Discussion}
In this paper we have developed and analysed in detail the simple thermal model for the development of a wax layer on the interior wall of a circular pipe transporting heated oil containing dissolved paraffinic wax, which was introduced in \cite{Needham}. 
This approach is gaining considerable traction compared to the traditional mechanical and material diffusion theories; it is able to describe features associated with wax layer formation which have been absent from, or even contrary to, the outcomes from the mechanical and material diffusion theories. 
This view point is vindicated in a number of recent detailed reviews of the wax layer formation process, see, for example \cite{M1}, \cite{MEH-SK1}, \cite{KL1} and \cite{vdG1}.
The current paper extends the theory developed in \cite{Needham} to allow for the dependence of solid wax thermal conductivity on local temperature, which is a significant feature of solidified paraffinic wax, and which should be included as a principle component in the thermal model. 
This inclusion modifies the free boundary problem formulated in \cite{Needham}, which now becomes, in general, a nonlinear free boundary problem. 
Despite the introduction of this fundamental nonlinearity, we have still been able to develop a comprehensive theory for this improved model. 
The variable conductivity model was developed in Section \ref{Model}, and formulated as a nonlinear free boundary problem [IBVP] in Section \ref{FBP}.
A detailed and rigorous qualitative theory for [IBVP] has been developed in Section \ref{FBP}, whilst the small and large time asymptotic structure to the solution to [IBVP] was given in Section \ref{CoordExp}. 
In Section \ref{kasymptotics} we developed the asymptotic solution to [IBVP] in the situation when the heat conductivity time scale in solid wax is much faster than the time scale associated with the inter-facial formation of solid wax. 
These substantial analytical developments have been supplemented by an efficient and readily implemented numerical scheme for [IBVP] in Section \ref{numerics}. 
Finally, a qualitative comparison between the present theory and the experimental results in \cite{HA1} was briefly considered in Section \ref{thevsexp}.
These initial comparisons are very encouraging for the thermal mechanism model introduced in \cite{Needham} and developed here. 
A more detailed experiment with model comparison to consider quantitative agreement appears to a worthwhile and appropriate endeavour to undertake.

\section*{Acknowledgements}
The authors would like to thank Ruben Schulkes (Statoil Petroleum, AS) for his helpful comments. 

\section*{Conflict of interest}
The authors declare that they have no conflict of interest.


\begin{thebibliography}{ww}
\bibitem{LAAT1}Azevedo LFA, Teixeira, AM (2003) A critical review of the modeling of wax deposition mechanisms. Petroleum Science and Technology 21:393–408. \href{https://doi.org/10.1081/LFT-120018528}{10.1081/LFT-120018528}

\bibitem{Needham2}Burger ED, Perkins TK, Striegler JH (1981) Studies of wax deposition in the trans Alaska pipeline. Journal of Petroleum Technology 33:1075–1086. \href{https://doi.org/10.2118/8788-PA}{10.2118/8788-PA}

\bibitem{Cannon1967}Cannon JR, Hill CD (1967) Existence, uniqueness, stability, and monotone dependence in a Stefan problem for the heat equation. J. Math. Mech 17:1-20. \href{https://doi.org/10.1512/iumj.1968.17.17001}{10.1512/iumj.1968.17.17001}

\bibitem{Needham5}Crank J (1984) Free and Moving Boundary Problems. Clarendon Press, Oxford.

\bibitem{Friedman1992}Friedman A (1982) Variational Principles and Free-Boundary Problems. John Wiley \& Sons, New York.

\bibitem{HAAH1}Halstensen M, Arvoh BK, Amundsen L, Hoffman R (2013) Online estimation of wax deposition thickness in single phase sub-sea pipelines based on acoustic chemometrics: a feasibility study. Fuel 105:718-727. \href{https://doi.org/10.1016/j.fuel.2012.10.004}{10.1016/j.fuel.2012.10.004}

\bibitem{HA1}Hoffman R, Amundsen L (2010) Single-phase wax deposition experiments. Energy and Fuels 24:1069-1080. \href{https://doi.org/10.1021/ef900920x}{10.1021/ef900920x}

\bibitem{KL1}Kaye GWC and Laby TH, Mechanical properties of materials, Tables of Physics Chemical constants, \textit{National Physical Laboratory (NPL)} (archived from the original on 11/03/2008).

\bibitem{LUS}Lady\u{z}enskaja OA, Solonnikov VA, Ural'ceva NN (1968) Linear and Quasi-Linear Equations of Parabolic Type. AMS, Rhode Island.

\bibitem{M1}Mahir LHA (2020) Modelling Paraffin Wax Deposition from Flowing Oil onto Cold Surfaces. Ph.D. Thesis, University of  Michigan.

\bibitem{MEH-SK1}Mehrotra AK, Ehsani S, Haj-Shaflei S, Kasuma S (2020) A review of heat-transfer mechanism for solid deposition from `waxy' or paraffinic mixtures. Canadian Journal of Chemical Engineering 98:12:2463-2488. \href{https://doi.org/10.1002/cjce.23829}{10.1002/cjce.23829}

\bibitem{Needham}Needham DJ, Johansson BT, Reeve T (2014) The development of a wax layer on the interior wall of a circular pipe transporting heated oil. QJMAM 67:93–125. \href{https://doi.org/10.1093/qjmam/hbt025}{10.1093/qjmam/hbt025}

\bibitem{Schatz1969}Schatz A (1969) Free boundary problems of Stephan type with prescribed flux. J. Math. Anal. Appl 28:569-580. \href{https://doi.org/10.1016/0022-247X(69)90009-2}{10.1016/0022-247X(69)90009-2}

\bibitem{Schulkes}Schulkes RMSM (2011) Modelling wax deposition as a Stefan problem. Unpublished notes.

\bibitem{SVSN1}Singh P, Venkatesan R, Scott Fogler H, Nagaragan NR (2000) Formation and aging of incipient thin film wax-oil gels, AICLE J. 46:5:1059-1074. \href{https://doi.org/10.1002/aic.690460517}{10.1002/aic.690460517}

\bibitem{SKSN2}Singh P, Venkatesan R, Scott Fogler H, Nagaragan NR (2000) Aging and morphological evolution of wax-oil gels during externally cooled flow through pipes. Second International Conference in Petroleum Phase Behaviour and Fouling. Copenhagen, Denmark.

\bibitem{vdG1}van der Geest C, Melchuna A, Bizarre L, Bannwart AC and Guersoni VCB (2021) Critical review on wax deposition in single-phase flow. Fuel 293:120358. \href{https://doi.org/10.1016/j.fuel.2021.120358}{10.1016/j.fuel.2021.120358}

\bibitem{WW}Walter W (1986) On the strong maximum principle for parabolic differential equations. Proc. Edinb, Math, Soc 29:93-96. \href{https://doi.org/10.1017/S0013091500017442}{10.1017/S0013091500017442}

\end{thebibliography}
\end{document}